\documentclass[12pt,draftcls,onecolumn]{IEEEtran}
\IEEEoverridecommandlockouts                              

\usepackage{amssymb}
\usepackage{epsfig}
\usepackage{psfrag}
\usepackage{latexsym}
\usepackage{amsmath}
\usepackage{amsfonts}
\usepackage{comment}
\usepackage{setspace}
\usepackage{epsfig}
\usepackage{dsfont}
\usepackage{txfonts}
\title{\LARGE \bf
An Optimal Controller Architecture for Poset-Causal Systems}

\author{Parikshit Shah and Pablo A. Parrilo  \\
MIT }
\normalsize
\def\T={\buildrel {\scriptscriptstyle\triangle} \over =}

\def\mc{\mathcal}
\def\mb{\mathbb}

\def\ba{\begin{array}}
\def\ea{\end{array}}
\def\ll{\left[}
\def\rr{\right]}
\def\l{\left\{}
\def\r{\right\} }
\def\up{{\uparrow}}

\def\up{{\uparrow}}
\def\down{{\downarrow}}

\def \inc {\mathcal{I}(\mathcal{P})}
\def \vec{\text{vec}}
\def\kron{\otimes}
\def\diag{\text{diag}}
\newtheorem{theorem}{Theorem}

\newtheorem{definition}{\indent Definition}

\newtheorem{lemma}{Lemma}
\newtheorem{example}{Example}
\newenvironment{remark}[1][Remark]{\begin{trivlist}
\item[\hskip \labelsep {\bfseries #1}]}{\end{trivlist}}
\newenvironment{remarks}[1][Remarks]{\begin{trivlist}
\item[\hskip \labelsep {\bfseries #1}]}{\end{trivlist}}

\begin{document}
\maketitle
\thispagestyle{plain}
\pagestyle{plain}

\bibliographystyle{plain}


\begin{abstract}
  We propose a novel and natural architecture for decentralized
  control that is applicable whenever the underlying system has the
  structure of a partially ordered set (poset). This controller
  architecture is based on the concept of M\"obius inversion for posets, and
  enjoys simple and appealing separation properties, since the
  closed-loop dynamics can be analyzed in terms of decoupled
  subsystems. The controller structure provides rich and interesting
  connections between concepts from order theory such as M\"{o}bius
  inversion and control-theoretic concepts such as state prediction,
  correction, and separability. In addition, using our earlier results
  on $\mathcal{H}_2$-optimal decentralized control for arbitrary
  posets, we prove that the $\mc{H}_2$-optimal controller in fact
  possesses the proposed structure, thereby establishing the
  optimality of the new controller architecture.
\end{abstract}

\section{Introduction}\label{sec:1}
The prevalence of large-scale complex systems in many areas of engineering has emphasized the need for a systematic study of decentralized control. While the problem of decentralized control in full generality remains a challenging task, certain classes of problems have been shown to be more tractable than others \cite{poset,quadinv,spatinv,poset2}.

Motivated by the intuition that \emph{acyclic} structures within the
context of decentralized control should be tractable, the authors
began a systematic study of a class of systems known as poset-causal
systems in \cite{poset}. In follow-up work \cite{pari_thesis,poset_h2}
we showed that the problem of computing $\mc{H}_2$-optimal controllers
using state-space techniques over this class of systems was tractable,
with efficient solutions in terms of uncoupled Riccati equations. We
also provided several intuitive explanations of the controller
structure, though a detailed analysis of the same was not presented.

In this paper we are concerned with the following questions:
``What is a sensible architecture of controllers for poset-causal
systems? What should be the role of controller states, and what
computations should be involved in the controller?'' This paper
focuses on answering this \emph{architectural} question. Our main
contributions in this paper are:
\begin{itemize}
 \item We propose a controller architecture that involves natural concepts from order theory  and control theory as building blocks. 
\item We show that a natural coordinate transformation of the state variables yields a novel \emph{separation principle}. 
\item We show that the optimal $\mc{H}_2$ controller (with state-feedback) studied in \cite{poset_h2} has precisely the proposed controller structure.
\item We establish novel connections that tie together three well-known concepts: (a) Youla parameterization in control, (b) the concept of purified output feedback in robust optimization and (c) M\"{o}bius inversion on posets.
\end{itemize}

\begin{figure}[htbp]
  \begin{center}
  \includegraphics[scale=0.4]{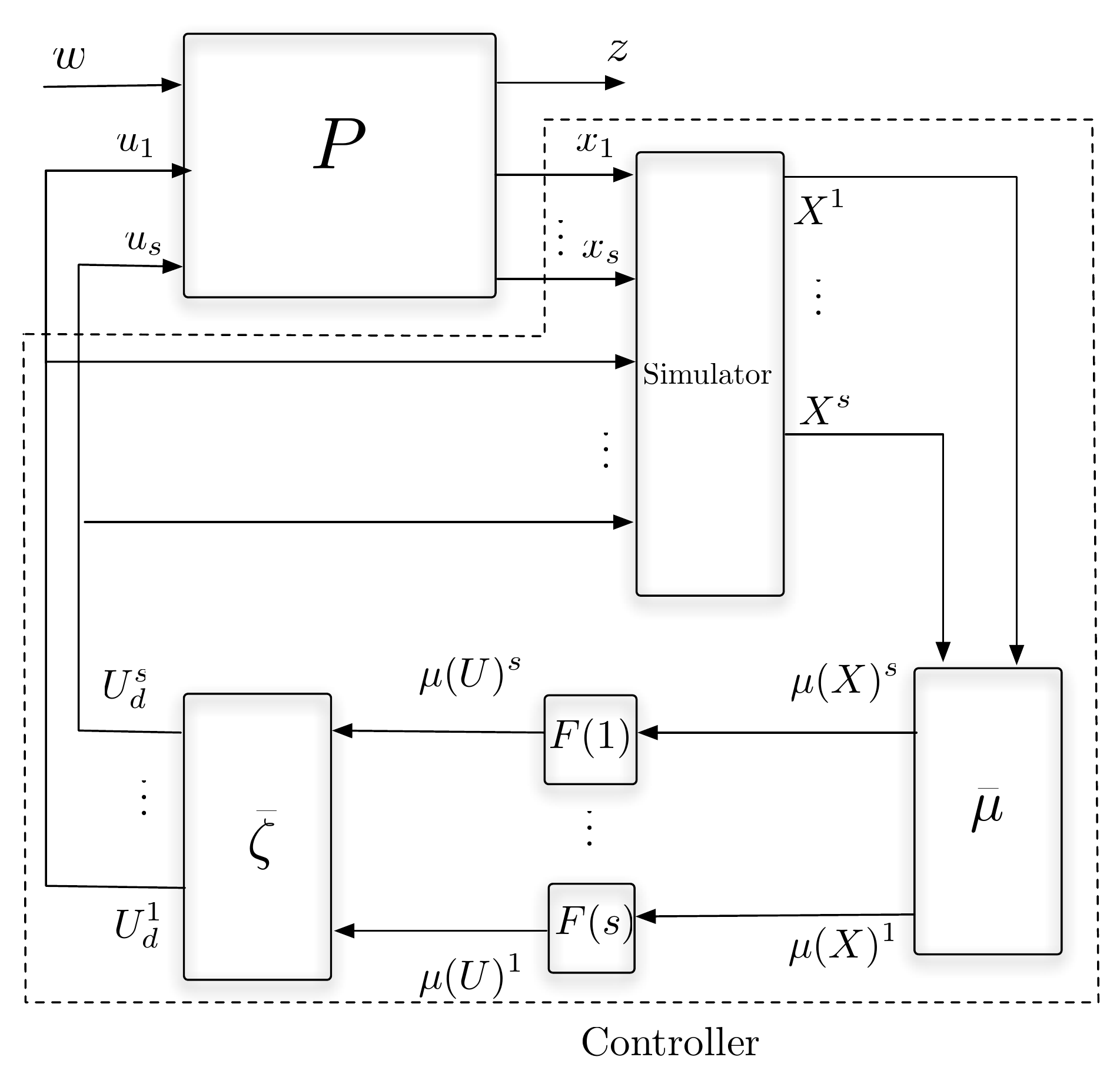}
  \end{center}
  \caption{A block-diagram representation of the control architecture. The simulator predicts the unknown states at each subsystem using available information (the prediction at subsystem $k$ is denoted by $X^k$). The controller then computes the differential improvement in the prediction using $\mu$, acts on it with a local gain $F(k)$ and then ``integrates'' the signal along the poset using $\zeta$ to produce the input signal.}
  \label{fig:tf1}
  \end{figure}

The controller structure that we propose in this paper is as follows.
At each subsystem of the overall system, the partial ordering of the
information structure allows one to decompose the global state into
``upstream'' states (i.e. states that are available), ``downstream''
(these are unavailable) and ``off-stream'' states (corresponding to
uncomparable elements of the poset). The downstream and off-stream
states are (partially) predicted using available upstream information using a ``simulator'' (see Fig. \ref{fig:tf1}),
this prediction is the role of the controller states. The best
available information of the global state at each subsystem is then
described using a matrix $X$; each column of $X$ corresponds to the
best local guess or estimate of the overall state at a particular
subsystem.

Having computed these local partial estimates, the controller then
performs certain natural local operations on $X$ that preserve the
structure of the poset. These local operations are the well-known
$\zeta$ and $\mu$ operations in M\"obius inversion. These operations,
which are intimately related to the inclusion-exclusion formula and
its generalizations, have a rich and interesting theory, and appear in
a variety of mathematical contexts \cite{rota}.  The control inputs
are of the form $U=\zeta(\mathbf{F} \circ \mu(X))$.  As we will see
later, the operators $\mu$ and $\zeta$ can be interpreted as
generalized notions of differentiation and integration on the poset so
that $\mu(X)$ may be interpreted as the differential improvement in
the prediction of the local state. Here $\mathbf{F}=\left\{F(1), \ldots, F(s) \right\}$ are feedback gain matrices corresponding to the different subsystems. The quantity $\mathbf{F} \circ
\mu(X)$ may therefore be interpreted as a local ``differential
contribution'' to the overall control signal. The overall control law
then aggregates all these local contributions by ``integration'' along
the poset using $\zeta$. This architecture has been shown diagrammatically in Fig. \ref{fig:tf1}.

Computational and architectural issues in decentralized control have
been important areas of study; we mention some related works below.
From a computational standpoint, the problem of computing
$\mc{H}_2$-optimal controllers for quadratically invariant systems was
studied in \cite{quadinv}, however that approach does not provide much
insight into the structure of the optimal controller.  In the context
of decentralized control, the computational and architectural issues
for the ``Two-Player Case'' were studied in \cite{Swigart2}. This work
was extended to arbitrary posets in \cite{poset_h2} (similar results
were obtained in \cite{swigart_thesis}), and some hints regarding the
structure of the optimal controller were provided in our previous
work. Another important related work is the simpler but related
\emph{team-theory} problem over posets studied in \cite{Ho} which
provides us with an interesting starting point in this paper. We mention the work of Witsenhausen \cite{Witsenhausen2,Witsenhausen1}
who provided important insight regarding different types of
information constraints in control problems. Finally, team theory and decentralized control have also been studied in \cite{gattami_thesis}.

The rest of this paper is organized as follows: In Section~\ref{sec:2}
we introduce the necessary order-theoretic and control-theoretic
preliminaries. In Section~\ref{sec:ingredients} we present the basic
building blocks involved in the controller architecture.  In
Section~\ref{sec:general_poset} we describe in detail the proposed
architecture, establish the separability principle and explain its
optimality property with respect to the $\mc{H}_2$ norm. In Section~\ref{sec:5}, we discuss a block diagram perspective to interpret our results. In Section~\ref{sec:6}, we discuss connections to the Youla parameterization and the literature on purified output feedback.


\section{Preliminaries}\label{sec:2}
In this section we introduce some concepts from order theory. Most of
these concepts are well-studied and fairly standard, we refer the
reader to \cite{aigner}, \cite{davey} for details.
\subsection{Posets}
\begin{definition}
A partially ordered set (or \emph{poset}) $\mc{P}=(P, \preceq)$ consists of a set $P$
along with a binary relation $\preceq$ which is reflexive, anti-symmetric and transitive \cite{aigner}.
\end{definition}
We will sometimes use the notation $a \prec b$ to denote the strict order relation $a \preceq b$ but $a \neq b$. 

An important related concept is that of a \emph{product} of two posets.
\begin{definition} \label{def:poset_product}
Let $\mc{P}=(P, \preceq_{\mc{P}})$ and $\mc{Q}=(Q,\preceq_{\mc{Q}})$ be two posets. We define their \emph{product} poset $\mc{P}\times \mc{Q}=(P \times Q,\preceq_{\mc{P}\times \mc{Q}}) $ to be the set $P \times Q$ equipped with the order relation $\preceq_{\mc{P}\times \mc{Q}}$ satisfying 
$$
(p_1,q_1) \preceq_{\mc{P}\times \mc{Q}} (p_2,q_2) \text{ if } p_1 \preceq_{\mc{P}} p_2 \text{ and } q_1 \preceq_{\mc{Q}} q_2.
$$
\end{definition}
It may be easily verified that $\preceq_{\mc{P}\times \mc{Q}}$ as defined above constitutes a partial order relation.

Most of this paper we will deal with finite posets (i.e. $|P|$ is finite).
It is possible to represent a poset graphically via a \emph{Hasse diagram} by representing the transitive reduction of the poset as a graph \cite{aigner}.
\begin{example} \label{example:1}
An example of a poset with three elements (i.e., $P=\left\{ 1,
2, 3 \right\}$) with order relations $1 \preceq 2$ and $1 \preceq 3 $
is shown in Figure \ref{fig:2.1}(b).
\begin{figure}[htbp]
  \begin{center}
  \includegraphics[scale=0.4]{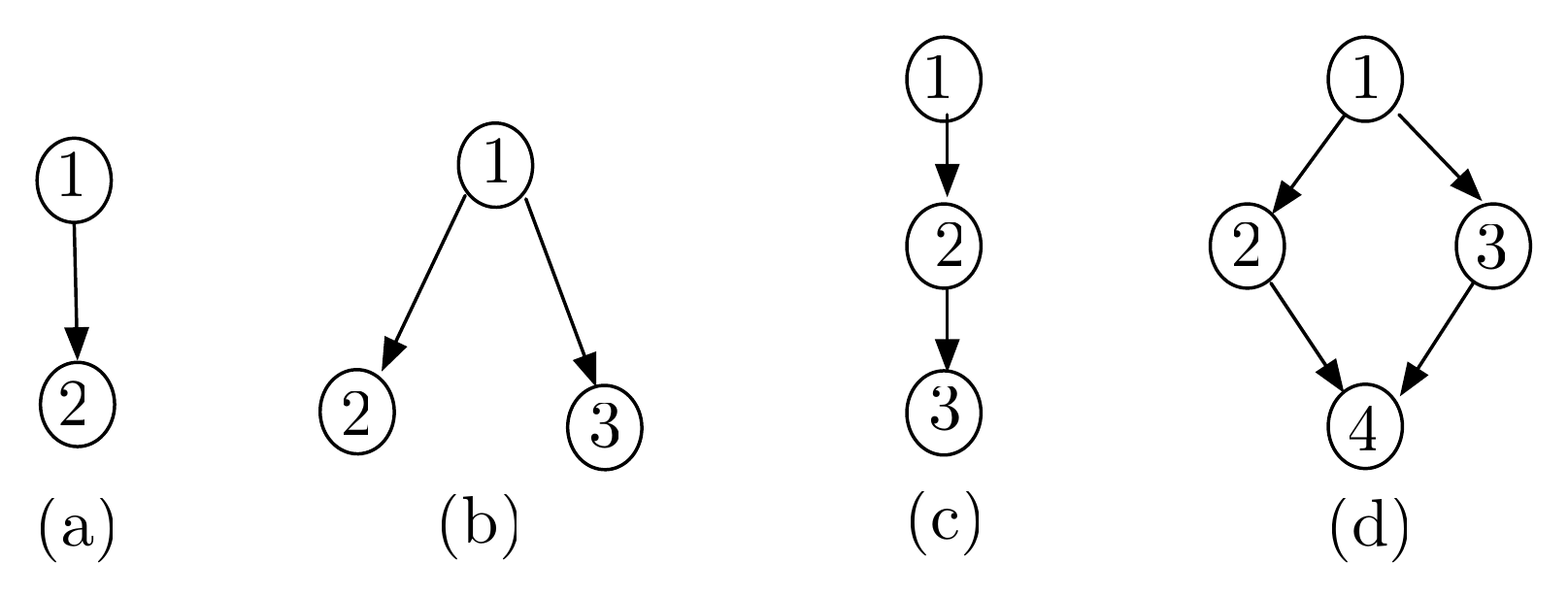}
  \end{center}
  \caption{Hasse diagrams of some posets.}
  \label{fig:2.1}
\end{figure}
\end{example}
Let $\mc{P}=(P, \preceq)$ be a poset and let $p\in P$. We define $\downarrow p=\l q\in P \; |\; p \preceq q \r$ (we call this the \emph{downstream set}).
Let $\down \down p = \l q\in P \; | \; p \preceq q, q \neq p \r$. 
Similarly, let  $\uparrow p=\l q\in P \; | \; q \preceq p \r$ (called a \emph{upstream set}), and $\up \up p= \l q\in P \; | \; q \preceq p, q \neq p \r$. We define $ \down \up p = \l q\in P \; | \; q \npreceq p, q \npreceq p \r$ (called the \emph{off-stream set}); this is the set of \emph {uncomparable} elements that have no order relation with respect to $p$. 
Define an \emph{interval} $[i,j] =\l p \in P \; | \; i \preceq p \preceq j \r$.
 A \emph{minimal element} of the poset is an element $p \in P$ such that if $q \preceq p$ for some $q \in P$ then $q=p$. (A maximal element is defined analogously).
 
 In the poset shown in Figure \ref{fig:2.1}(d), $\down 1 = \l 1, 2, 3, 4 \r$, whereas $\down \down 1 = \l  2, 3, 4 \r$. Similarly $\up \up 1 = \emptyset$, $\up 4 = \l 1, 2, 3, 4 \r$, and $\up \up 4 = \l 1, 2, 3 \r$. The set $\down \up 2=\l 3 \r$. 
\begin{definition} \label{def:inc_algebra}
Let $\mc{P}=(P,\preceq)$ be a poset. Let $\mathbb{Q}$ be a ring. The set of all functions
$f:P \times P \rightarrow \mathbb{Q}$
with the property that $f(x,y)=0$ if $y \npreceq x$ is called the
\emph{incidence algebra} of $\mc{P}$ over $\mathbb{Q}$. It is denoted by
$I(\mc{P})$. 
\footnote[1]{Standard definitions of the incidence algebra use the opposite convention, namely $f(x,y)=0$ if $x \npreceq y$ so their matrix representation typically has upper triangular structure. We reverse the convention so that they are lower-triangular, and thus in a control-theoretic setting one may interpret them as representing \emph{poset-causal} maps. This reversal of convention entails transposing other standard objects like the zeta and the M\"{o}bius operators. For the same reason, we also reverse the convention of drawing Hasse diagrams so that minimal elements appear at the top of the poset.}
\end{definition}
When the poset $\mc{P}$ is finite, the elements in the incidence
algebra may be thought of as matrices with a specific sparsity
pattern given by the order relations of the poset in the following way. 
An example of an element of $\mc{I}(\mc{P})$ for the poset from
Example 1 (Fig. \ref{fig:2.1}(b)) is:
\begin{equation*}
\zeta_{\mc{P}}=\left[ \begin{array}{ccc} 1 & 0 & 0 \\
1 & 1 & 0 \\
1 & 0 & 1 \end{array} \right].
\end{equation*}
Given two functions $f, g \in I(\mc{P})$, their sum $f+g$ and scalar
multiplication $c f$ are defined as usual. The product $h=f \cdot g$
is defined by $h(x,y)=\sum_{z \in P}f(x,z)g(z,y)$.  Note that the
above definition of function multiplication is made so that it is
consistent with standard matrix multiplication. It is well-known that
the incidence algebra is an associative algebra \cite{aigner,poset}.
\subsection{Control Theoretic Preliminaries}
\subsubsection{Poset-causal systems}
We consider the following state-space system in continuous time:

\begin{equation} \label{eq:2}
\begin{split}
\dot{x}(t)&=Ax(t)+w(t)+Bu(t) \\
z(t)&=Cx(t)+Du(t) \\
y(t)&=x(t).
\end{split}
\end{equation}
In this paper we present the continuous time case only, however, we wish
to emphasize that analogous results hold in discrete time in a
straightforward manner.  In this paper we consider what we will call
\emph{poset-causal systems}. We think of the system matrices
$(A,B,C,D)$ to be partitioned into blocks in the following natural
way. Let $\mc{P}=(P,\preceq)$ be a poset with $P=\left\{1,\ldots, s
\right\}$.  We think of this system as being divided into $s$
subsystems, with subsystem $i$ having some states $x_i(t) \in
\mb{R}^{n_i}$, and we let $N=\sum_{i \in P} n_i$ be the total degree
of the system. The control inputs at the subsystems are $u_i(t) \in
\mb{R}^{m_i}$ for $i\in \left\{1,\ldots, s \right\}$.  The external
output is $z(t) \in \mb{R}^{p}$. The signal $w(t)$ is a disturbance
signal. The states and inputs are partitioned in the natural way such
that the subsystems correspond to elements of the poset $\mc{P}$ with
$x(t)=\left[ x_1(t)\left| x_2(t) \left| \ldots \left| x_s(t)
      \right. \right. \right. \right] ^{T}$, and $u(t)=\left[
  u_1(t)\left| u_2(t) \left| \ldots \left| u_s(t)
      \right. \right. \right. \right] ^{T}$. This naturally partitions
the matrices $A, B, C, D$ into appropriate blocks so that
$A=\left[A_{ij} \right]_{i,j \in P}$, $B=\left[B_{ij} \right]_{i,j \in
  P}$, $C=\left[C_{j} \right]_{j \in P}$ (partitioned into columns),
$D=\left[D_{j} \right]_{j \in P}$. (We will throughout deal with
matrices at this block-matrix level, so that $A_{ij}$ will
unambiguously mean the $(i,j)$ block of the matrix $A$.) Using these
block partitions, one can define the incidence algebra at the block
matrix level in the natural way.  The block sizes will be obvious from
the context and we denote by $\mc{I}(\mc{P})$ the block incidence
algebra.
\begin{remark}
  In this paper, for notational simplicity we will assume $n_i=1$, and
  $m_i=1$. We emphasize that this is only done to simplify the
  presentation; the results hold for arbitrary block sizes $n_i$ and
  $m_i$ by interpreting the formulas ``block-wise'' in the obvious way.
\end{remark} 

We call such systems \emph{poset-causal} due to the following causality-like property among the subsystems. If an input is applied to subsystem $i$ via $u_i$ at some time $t$, the effect of the input is seen by the downstream states $x_j$ for all subsystems $j \in \downarrow i$ (at or after time $t$). Thus $ \downarrow i$ may be seen as the cone of influence of input $i$. We refer to this causality-like property as \emph{poset-causality}. This notion of causality enforces (in addition to causality with respect to time), a causality relation between the subsystems with respect to a poset.

\subsubsection{Information Constraints on Controller}
 In this paper, we will be interested in the design of poset-causal controllers of the form:
\begin{equation}\label{eq:controller_form}
K=\ll \ba{c|c}
A_K & B_K \\ \hline
C_K & D_K
\ea
\rr.
\end{equation}
We will require that the controller also be poset-causal, i.e. that $K \in \mc{I}(\mc{P})$. In later sections we will present a general architecture for controllers with this structure with some elegant properties. 


A control law \eqref{eq:controller_form} with $K \in \inc$ is said to be \emph{poset-causal}
since $u_i$ depends only on $x_j$ for $j \in \uparrow i$ (i.e. upstream information) thereby enforcing poset-causality constraints also on the controller.
\vspace{-15pt}
\subsection{Notation}
Since we are dealing with poset-causal systems (with respect to the poset $\mc{P}=(P, \preceq))$, most vectors and matrices will be naturally indexed with respect to the set $P$ (at the block level). Recall that every poset $\mc{P}$ has a linear extension (i.e. a total order on $P$ which is consistent with the partial order $\preceq$). For convenience, we fix such a linear extension of $\mc{P}$, and all indexing of our matrices throughout the paper will be consistent with this linear extension (so that elements of the incidence algebra are lower triangular).

Given a matrix $M$, $M_{ij}$ will as usual denote the $(i,j)^{th}$ entry. 
The $i^{th}$ column will be denoted by $M^i$. If M is a block $|P| \times |P|$ matrix, we will denote $M(\down i, \down i)$ to be the sub-matrix of $M$ whose rows and columns are in $\down i$. We will also need to deal with the inverse operation: we will be given an $|S| \times |S|$ matrix $K$ (indexed by some subset $S \subseteq P$) and we will wish to embed it into a $|P| \times |P|$ matrix by zero-padding the locations corresponding to row and column locations in $P \setminus S$. We will denote this embedded matrix by $\hat{K}$.

\section{Ingredients of the Architecture} \label{sec:ingredients}
The controller architecture that we propose is composed of three main ingredients:
\begin{itemize}
\item The notion of \emph{local variables},
\item A notion of a local product, denoted by ``$\circ$'',
\item A pair of operators $\zeta, \mu$ that operate on the local variables in a way that is consistent with the order-theoretic structure of the poset. These operators, called the \emph{zeta} operator and the \emph{M\"{o}bius} operator respectively, are classical objects and play a central role in much of order theory, number theory and combinatorics \cite{rota}.
\end{itemize}
\vspace{-15pt}
\subsection{Local Variables and Local Products}
We begin with the notion of global variables. 
\begin{definition}
A function $Z:P \times P \rightarrow \mathbb{R}$ is called a \emph{local variable}.
A function $z:P \rightarrow \mathbb{R}$ is called a \emph{global variable}. The local variable $Z$ is said to be \emph{consistent} with the global variable $z$ if $Z(i,i)=z(i)$ for all $i \in P$.
\end{definition}
\begin{remark}
When the set $P$ is finite it is convenient to think of local variables $Z$ as matrices in $\mathbb{R}^{s \times
    s}$ and global variables $z$ as vectors in $\mathbb{R}^{s}$. The local variable $Z$ is consistent with the global variable $z$ if 
 $Z_{ii}=z_i$.
\end{remark}
Typical global variables that we encounter will be the overall state $x$ and the input $u$. 
Note that the overall system is composed of $s=|P|$ subsystems. Subsystem $i$ has access to components of the global variable corresponding to $\up i$, and components corresponding to $\down \down i$ are unavailable. One can imagine each subsystem maintaining a local prediction of the global variable. This notion is captured by the following. 
The $i^{th}$ column of $Z$, denoted by $Z^i$ is to be thought of as a local prediction of $z$ at subsystem $i$. The components corresponding to $\down \down i$ correspond to the predictions of the unknown (downstream) components of $z$. Note that $Z_{ii}=z_i$ so that at subsystem $i$ the component $z_i$ of the global variable is available.

We will use the indexing $Z^{i}=[Z^{i}_j]_{j \in P}$, so that $Z^{i}_j$ denotes the local prediction of $z_j$ at subsystem $i$. We will sometimes also denote $Z^{i}_j$ by $z_j(i)$. While local variables in general are full matrices, an important class of local variables that we will encounter will have the property that they are in $\inc$.
The two important local variables we will encounter are $X$ (local state variables) and $U$ (local input variables).
\begin{example}\label{example:local_state}
We illustrate the concepts of global variables and local variables  with an example. Consider the poset shown in Fig. \ref{fig:2.1}(d). Then we can define the global variable $x$ and a corresponding local variable $X$ as follows: \small
\begin{align*}
x=\left[ \ba{c}
x_1 \\ x_2 \\ x_3 \\ x_4
\ea \right] & & &
X=\left[ \ba{cccc}
x_1 & x_1 & x_1 & x_1 \\
x_2(1) & x_2 & x_2(1) & x_2 \\
x_3(1) & x_3(1) & x_3 & x_3 \\
x_4(1) & x_4(2) & x_4(3) & x_4
\ea \rr. \normalsize
\end{align*}
\end{example}

We define the following important product:
\begin{definition} \label{def:product}
Let $\mathbf{F}=\left\{ F(1), \ldots, F(s) \right\}$ be a collection of maps 
$F(i): \down i \times \down i \rightarrow \mathbb{R}$ (viewed as matrices).
Let $X$ be a local variable. We define the \emph{local product} $\mathbf{F}\circ X$ columnwise via
\begin{equation}\label{eq:product}
(\mathbf{F}\circ X)^{i} \triangleq \hat{F}(i)X^{i} \; \; \; \text{for all } i \in P. 
\end{equation}
\end{definition}
\begin{remark}
Note that if $X \in \inc$ and $Y=\mathbf{F} \circ X$, then it is easy to verify that $Y \in \inc$. 
We call the matrices $F(i)$ the \emph{local gains}.
Local products give rise to decoupled local relationships in the following natural way. Let $X, Y$ be local variables. 
If they are related via $Y=\mathbf{F} \circ X$ then the relationship between $X$ and $Y$ is said to be \emph{decoupled}. This is because, by definition,
$$
Y^{k}=\hat{F}(k) X^{k} \; \text{ for all } k \in P.
$$
Thus the maps relating the pairs $(X^{k}, Y^{k})$ are \emph{decoupled} across all $k \in P$ (i.e. $Y^{k}$ depends only on $X^{k}$ and not on $X^{j}$ for any other $j \neq k$). 
\end{remark}
\begin{example}
Continuing with Example \ref{example:local_state}, let us define the local gains by 
$
\mathbf{F}= \left\{ F(1), F(2), F(3), F(4) \right\},
$
where,
\begin{align*}
F(1)&= \ll \ba{cccc}
F_{11}(1) & F_{12}(1) & F_{13}(1) & F_{14}(1) \\
F_{21}(1) & F_{22}(1) & F_{23}(1) & F_{24}(1) \\
F_{31}(1) & F_{32}(1) & F_{33}(1) & F_{34}(1) \\
F_{41}(1) & F_{42}(1) & F_{43}(1) & F_{44}(1) 
\ea \rr & & 
F(2)&= \ll \ba{cccc}
0 & 0 & 0 & 0 \\
0 & F_{22}(2) & 0& F_{24}(2) \\
0 & 0 & 0 & 0 \\
0 & F_{42}(2) & 0 & F_{44}(2) 
\ea \rr \\ 
F(3)&= \ll \ba{cccc}
0 & 0 & 0 & 0 \\
0 & 0 & 0 & 0 \\
0 & 0 & F_{33}(3) & F_{34}(3) \\
0 & 0 & F_{43}(3) & F_{44}(3) 
\ea \rr & & 
F(4)&= \ll \ba{cccc}
0 & 0 & 0 & 0 \\
0 & 0 & 0 & 0 \\
0 & 0 & 0 & 0 \\
0 & 0 & 0 & F_{44}(4) 
\ea \rr. 
\end{align*}
Then 
\begin{align*}
\mathbf{F}\circ X=\left[ \ba{cccc} 
F(1)\ll \ba{c} X_{11} \\ X_{21} \\ X_{31} \\ X_{41} \ea \rr &
F(2)\ll \ba{c} 0 \\ X_{22} \\ 0 \\ X_{42} \ea \rr &
F(3)\ll \ba{c} 0 \\ 0 \\ X_{33} \\ X_{43} \ea \rr &
F(4)\ll \ba{c} 0 \\ 0 \\ 0 \\ X_{44} \ea \rr 
\ea \right].
\end{align*}
\end{example}
\begin{definition} \label{def:matrix_mult}
Let $M \in \mathbb{R}^{s \times s}$ be a matrix. Define
\begin{align*}
\Pi_d(M)=\left\{\ba{c} M_{ij} \; \; \text{for } i \preceq j\\
0 \; \; \text{otherwise.}
\ea \right. & \qquad
\Pi_{{uo}}(M)=\left\{\ba{c} M_{ij} \; \; \text{for } i \npreceq j\\
0 \; \; \text{otherwise.}
\ea \right.
\end{align*}
\end{definition}
Thus the matrix $M$ can be decomposed as
$$
M=M_d+M_{uo}.
$$
 The component $M_d=\Pi_d(M)$ simply corresponds to the ``downstream component'', and is the projection of the matrix $M$ onto the incidence algebra $\inc$ viewed as a subspace of matrices. The component  $M_{uo}=\Pi_{{uo}}(M)$ corresponds to the ``upstream and offstream elements'' and is the projection onto the orthogonal complement.
%

\subsection{The M\"{o}bius and zeta operators}

We first remind the reader of two important order-theoretic notions,
namely the \emph{zeta and M\"{o}bius operators}. These are well-known
concepts in order theory that generalize discrete integration and
finite differences (i.e. discrete differentiation) to posets.
\begin{definition} 
\label{def:zeta} 
Let $\mc{P}=(P,\preceq)$. The \emph{zeta matrix} $\zeta $ is defined
to be the matrix $\zeta:P \times P \rightarrow \mathbb{R}$ such that
$\zeta(i,j)=1$ whenever $j \preceq i$ and zeroes elsewhere. The
\emph{M\"{o}bius matrix} is its inverse, $\mu:=\zeta^{-1}$.
\end{definition}
These matrices may be viewed as operators acting on functions on the poset $f:P \rightarrow \mathbb{R}$ (the functions being expressed as row vectors). The matrices $\zeta, \mu$, which are members of the incidence algebra, act as linear transformations on $f$ in the following way:
\begin{align*}
\zeta: &\mathbb{R}^{|P|} \rightarrow \mathbb{R}^{|P|} & & \mu: \mathbb{R}^{|P|} \rightarrow  \mathbb{R}^{|P|} \\
& f \mapsto f \zeta^{T} & &  f \mapsto f \mu^{T}.
\end{align*}
Note that $\zeta(f)$ is also a function on the poset given by 
\begin{equation} \label{eq:zeta1}
\left( \zeta(f) \right)_{i}=\sum_{j \preceq i} f_j.
\end{equation} 
This may be naturally interpreted as a discrete integral of the function $f$ over the poset.

The role of the M\"{o}bius operator is the opposite: it is a generalized finite difference (i.e. a discrete form of differentiation over the poset). If $f:P \rightarrow \mathbb{R}$ is a local variable then the function $\mu(f): P \rightarrow \mathbb{R}$ may be computed recursively by:
\begin{equation} \label{eq:mobius0}
\left( \mu(f) \right)_i=\left\{ \ba{l} f_i \text{ for } i \text{ a minimal element,} \\
f_i-\sum_{j \prec i } \left(\mu(f)\right)_j \text{ otherwise. } \ea \right.
\end{equation}

\begin{example}
Consider the poset in Figure \ref{fig:2.1}(c). The zeta and the M\"{o}bius matrices are given by:
\begin{align*}
\zeta&=\ll \ba{ccc} 1 & 0 & 0 \\ 1 & 1 & 0 \\ 1 & 1 & 1 \ea \rr &
\mu&=\ll \ba{ccc} 1 & 0 & 0 \\ -1 & 1 & 0 \\ 0 & -1 & 1 \ea \rr.
\end{align*}
If $f=\ll \ba{ccc} f_1 & f_2 & f_3\ea \rr$, then
\begin{align*}
\zeta(f)&=\ll \ba{ccc} f_1 & f_1+f_2 & f_1+f_2+f_3 \ea \rr \\
\mu(f)&=\ll \ba{ccc} f_1 & f_2-f_1 & f_3-f_2 \ea \rr.
\end{align*}
\begin{figure}[htbp]
  \begin{center}
  \includegraphics[scale=0.35]{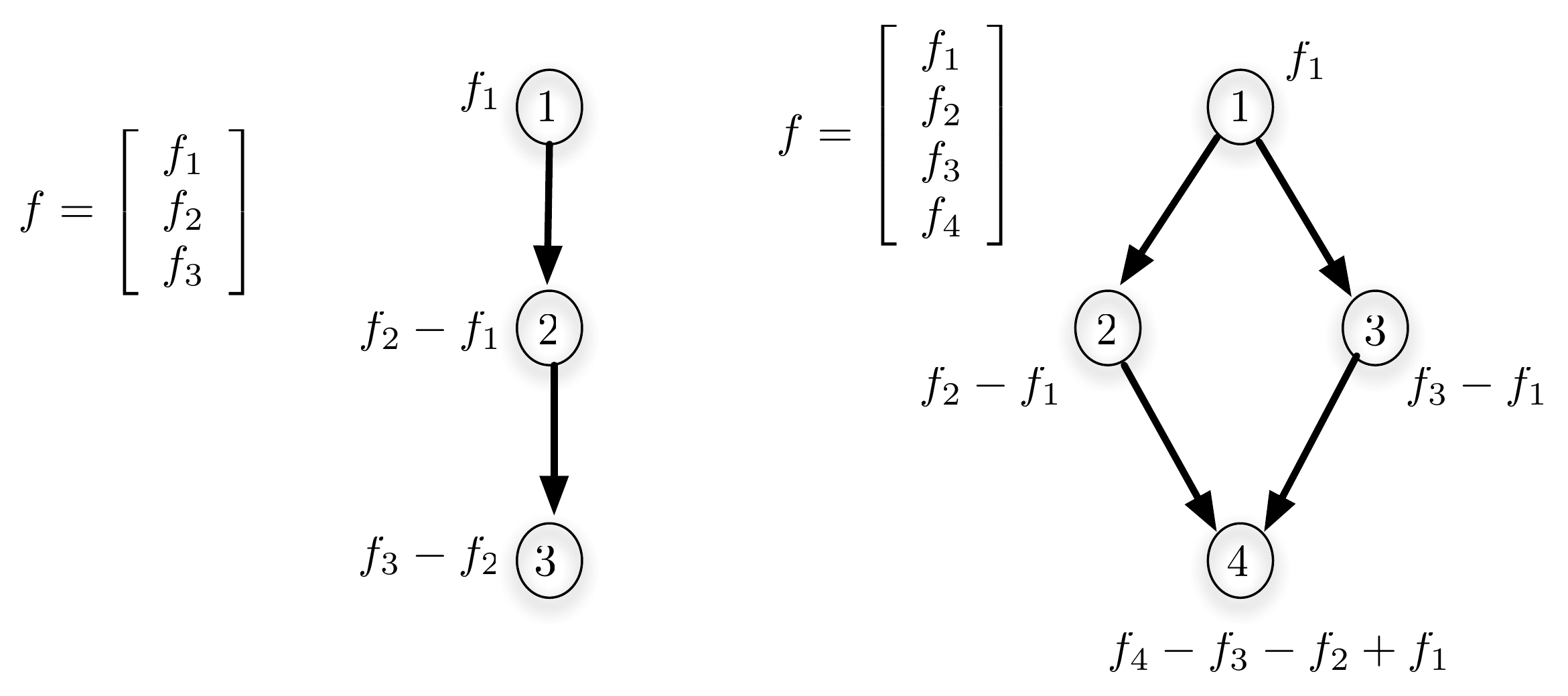}
  \end{center}
  \caption{Two posets with their M\"{o}bius operators. The functions $f$ are functions on the the posets, and the values of $\mu(f)$ at element $i$ are indicated next the the relevant elements.}
  \label{fig:2}
\end{figure}
\end{example}

We now define modified versions of the zeta and M\"{o}bius operators that extend the actions of $\mu$ and $\zeta$ from global variables $x$ to local variables $X$. Let $\zeta$ and $\mu$ be matrices as defined in Definition \ref{def:zeta}.
\begin{definition}\label{def:mobius_op}
Let $X$ be a local variable. 
Define the operators $\mu:\mathbb{R}^{s \times s} \rightarrow \inc$ and $\zeta: \mathbb{R}^{s \times s} \rightarrow \inc$ acting via
\begin{align}\label{eq:mobius}
\zeta(X)=\Pi_d(X\zeta^{T}) & & & \mu(X)=\Pi_d(X \mu^{T}).
\end{align}
\end{definition}
\begin{lemma} \label{lemma:mu1}
The operators $\zeta$ and $\mu$ may be written more explicitly as
\begin{align}\label{eq:mobius2}
\zeta(X)^{i}_j &\triangleq \sum_{k \preceq i} X^{k}_j  & & &\mu(X)^{i}_j &\triangleq X^{i}_j- \sum_{k \prec i} \mu(X)^{k}_j
\end{align}
for $i \preceq j$ and $0$ otherwise.
\end{lemma}
\begin{proof}
The proofs follow in a straightforward fashion from \eqref{eq:zeta1} and \eqref{eq:mobius0}.
\end{proof}
Note that if $Y=\mu(X)$ then $Y$ is a local variable in $\inc$.
The operator $\zeta$ has the natural interpretation of \emph{aggregating or integrating} the local variables $X^k$ for $k \in P$, whereas $\mu$ performs the inverse operation of differentiation of the local variables. 
\begin{example} \label{example:mu}
We illustrate the action of  $\mu$ acting on a local variable. Consider the local variable $X$ from Example \ref{example:local_state}. 
It is easy to verify that 
\begin{align*}
\scriptsize
\mu(X)= \ll \ba{cccc}  
x_1 & 0 & 0 & 0 \\
x_2(1) & x_2-x_2(1) & 0 & 0 \\
x_3(1) & 0 & x_3-x_3(1) & 0 \\
x_4(1) & x_4(2)-x_4(1) & x_4(3)-x_4(1) & x_4-x_4(3)-x_4(2)+x_4(1)
\ea \rr.
\normalsize
\end{align*}
\end{example}

\begin{lemma} \label{lemma:properties}
The operators $(\mu, \zeta)$ satisfy the following properties:
\begin{enumerate}
\item $(\mu, \zeta)$ are invertible restricted to $\inc$ and are inverses of each other so that for all local variables $X \in \inc$,
$$
\zeta(\mu(X))=\mu(\zeta(X))=X.
$$
\item $\mu(X)=\mu(\Pi_d(X))$ and $\zeta(X)=\zeta(\Pi_d(X))$. 
\item Let $A,X \in \inc$. Then $\mu(AX)=A\mu(X)$, and $\zeta(AX)=A\zeta(X)$.
\end{enumerate}
\end{lemma}
\begin{proof}
\begin{enumerate}
\item The proof is by induction. 
For a minimal element $i$, it is clear from Lemma \ref{lemma:mu1} that $\mu(\zeta(X))^i_j=\zeta(X)^i_j=X^i_j$. Now suppose $i$ is non-minimal. As the induction hypothesis, suppose the statement $\mu(\zeta(X))^k_j=X^k_j$ is true for all $k \prec i$. We prove the assertion is true for $i$. By Lemma \ref{lemma:mu1}, we have 
\begin{align*}
\mu(\zeta(X))^i_j&=\zeta(X)^i_j - \sum_{k \prec i} \mu(\zeta(X))^k_j \\
&=\zeta(X)^i_j-\sum_{k \prec i} X^k_j \; \; \text{(by induction hypothesis)} \\
&=X^i_j. 
\end{align*}
The proof that $\zeta(\mu(X))=X$ is similar.
\item For $j \npreceq i$, $\mu(X)^i_j=\mu(\Pi_d(X))^i_j=0$. Using Lemma \ref{lemma:mu1}, for $j \preceq i$, it is routine to check inductively that $\mu(X)^i_j$ depends only on the values of $X^k_j$ for $k \preceq i$. Hence $\mu(X)=\mu(\Pi_d(X) )$. The proof that $\zeta(\Pi_d(X))=\zeta(X)$ is similar.

\item Note that $\mu(X) \in \inc$, hence $A\mu(X) \in \in$, and hence $(A\mu(X))^{i}_j=\mu(A(X))^i_j=0$ for $j \npreceq i$. Now note that $\mu(AX)=\Pi_d(AX\mu^T)$, so that for $j \preceq i$,  
\begin{align*}
\mu(AX)^i_j&=\sum_{j \preceq k \preceq i}A_{ik}(X\mu^T)^k_j 
=\sum_{j \preceq k \preceq i}A_{ik}(\mu(X))^k_j 
=A\mu(X)^k_j.
\end{align*}
The proof that $\zeta(AX)=A\zeta(X)$ is similar.
\end{enumerate}
\end{proof}
Note that if $X \notin \inc$ then $\zeta(\mu(X))=\Pi_d(X)$. The second part of the preceding lemma says that $\mu(X)$ and $\zeta(X)$ \emph{depend only on the components of $X$ that lie in $\inc$}, i.e. on $\Pi_d(X)$. 

Since $\zeta$ and $\mu$ may be interpreted as integration and differentiation operators, the first part of the above lemma may be viewed as a ``poset''  version of the fundamental theorem of calculus. 

\section{Proposed Architecture} \label{sec:general_poset}
\subsection{Local States and Local Inputs}
Having defined local and global variables, we now specialize these concepts to our state-space system \eqref{eq:2}. We will denote $x_j$ to be the true state at subsystem $j$. We denote $x_j(i)$ to be a prediction of state $x_j$ at subsystem $i$. The information constraints at subsystem $i$ can be decoupled into three distinct cases: 
\begin{itemize}
\item Information about $\down i$: At subsystem $i$ the state information for $\down \down i$ unavailable, so a (possibly partial) prediction of $x_j$ for $j \in \down \down i$ is formed. We denote this prediction by $x_j(i)$. \emph{Computing these partial predictions is the role of the controller states}. The state $x_i$ is known at subsystem $i$. In the following discussion we will call the downstream predictions (as well as the true state $x_i$) at subystem $i$ the ``free variables in the architecture''.
 \item Information about $\up \up i$: Complete state information about $x_j$ for $j \in \up i$ is available, so that $x_j(i)=x_j$. Moreover, the predictions from upstream $x_k(j)$ for all $k \in P$ and $j \preceq i$  are also available. 
\item Information about $\down \up i$: At subsystem $i$, state information about $x_j$ for $j$ not comparable to $i$ is unavailable. The prediction of $x_j$ is computed using $x_j(k)$ for $k \prec i$. 
\end{itemize}
Analogous information constraints hold also for the inputs. At a particular subsystem, information about downstream inputs is not available. Consequently, we introduce the notion of prediction of unknown inputs, with similar notation as that for the states.
These ideas can be formalized by defining local variables that capture the best available information at the subsystems.
We introduce two local variables:
\begin{enumerate}
\item The local state $X$ associated with the system state $x$,
\item The local input $U$ associated with the controller input $u$.
\end{enumerate}
\begin{definition} \label{def:local_state}
The local state $X$ is a local variable that is consistent with the global state $x$ (i.e. $X^i_i=x_i$ for all $i \in P$) and satisfies the following additional properties.
\begin{enumerate}
\item $X_d:=\Pi_d(X)$ are free variables
\item The local variable $X$ and its component $X_{uo}:=\Pi_{uo}(X)$ are determined from $X_d$ via 
\begin{equation} \begin{split} \label{eq:X_{uo}1}
X_{uo}&=\Pi_{uo}(\mu(X_d) \zeta^T) \\ 
 X&=X_d+X_{uo}=\mu(X_d) \zeta^{T}.
\end{split}
\end{equation}
\end{enumerate}
\end{definition}
Note that the second equation in \eqref{eq:X_{uo}1} follows from the first and the fact that $X=X_d+X_{uo}$. This can be seen by noting that 
\begin{equation*}
X_d+\Pi_{uo}(\mu(X_d)\zeta^T)=\mu(\zeta(X_d))+\Pi_{uo}(\mu(X_d)\zeta^T)=\Pi_d(\mu(X_d)\zeta^T)\Pi_{uo}(\mu(X_d)\zeta^T)=\mu(X_d)\zeta^T.
\end{equation*}

We have an analogous definition for the local input:
\begin{definition}
The local input $U$ is a local variable that is consistent with the global input $u$ (i.e. $U^i_i=u_i$ for all $i \in P$) and satisfies the following additional properties.
\begin{enumerate}
\item $U_d:=\Pi_d(U)$ are free variables
\item The local variable $U$ and its component $U_{uo}=\Pi_{uo}(U)$ are determined from $U_d$ via 
\begin{equation}  \begin{split}\label{eq:U_{uo}}
U_{uo}&=\Pi_{uo}(\mu(U_d) \zeta^T) \\
U&=U_d+U_{uo}=\mu(U_d) \zeta^T. \end{split}\end{equation}
\end{enumerate}
\end{definition}
\begin{remark}
Definition \ref{def:local_state} has been carefully made so that the local variable $X$ enjoys some desirable properties.
At any subsystem, the local state $X^i$ can be decomposed into the downstream and the upstream/offstream components. We mention the properties below:
\begin{itemize}
\item The collection of downstream components of $X$ are the free variables $X_d$. 
\item The upstream and offstream components of $X$ are determined from the free variables $X_d$. Indeed, the reader may verify that as a consequence of \eqref{eq:X_{uo}1}, if $j \prec i$ then $X^i_j=x_j=(X_d)^j_j$ (this is intuitive since at subsystem $i$, $x_j$ is known if $j \prec i$). The predictions of offstream states are also computed using $X_d$. Equation \eqref{eq:X_{uo}1} implies that for offstream states (i.e. $i$ and $j$ uncomparable) $X^i_j=\sum_{k \prec i, k \prec j} \mu(X)^k_j$.
\item Our setup also ensures that $\mu(X)^i_j=0$ if $j \prec i$ or if $i$ and $j$ are incomparable. (This is because from Lemma \ref{lemma:properties}, $\mu(X)=\mu(X_d) \in \inc$). We will later see that this has a natural interpretation.
\end{itemize}
\end{remark}
\begin{example}
  Consider the poset shown in Fig. \ref{fig:2.1}(d). The matrix $X$
  shown in Example \ref{example:local_state} is a local state
  variable. The predicted partial states are $x_2(1), x_3(1), x_4(1),
  x_4(2), x_4(3)$. The plant states are $x_1, x_2, x_3, x_4$. These collectively are the free variables $X_d$. The components $X_d$ and $X_{uo}$ are given by
\begin{align*}
X_{d}=\left[ \ba{cccc}
x_1 & 0 &0 & 0 \\
x_2(1) & x_2 & 0 & 0 \\
x_3(1) & 0 & x_3 & 0 \\
x_4(1) & x_4(2) & x_4(3) & x_4
\ea \rr \qquad
X_{uo}=\left[ \ba{cccc}
0 & x_1 & x_1 & x_1 \\
0 & 0 & x_2(1) & x_2 \\
0 & x_3(1) & 0 & x_3 \\
0 & 0 & 0 & 0
\ea \rr.
\end{align*}
The reader may verify that $X_{uo}$ satisfies \eqref{eq:X_{uo}1}.
  Note
  that since subsystems $1$ and $2$ have the same information about
  subsystem $3$ ($2$ and $3$ are unrelated in the poset), the best
  estimate of $x_3$ at subsystem $2$ is $x_3(1)$.
\end{example}

\subsection{Role of $\mu$}
We now give a natural interpretation of the operator $\mu(X)$ in terms of the differential improvement in predicted states with the help of an example.
\begin{example}
Consider the poset shown in Fig.~\ref{fig:diff_imp}, and let us inspect the predictions of the state $x_5$ at the various subsystems.
\begin{figure}[htbp]
  \begin{center}
  \includegraphics[scale=0.5]{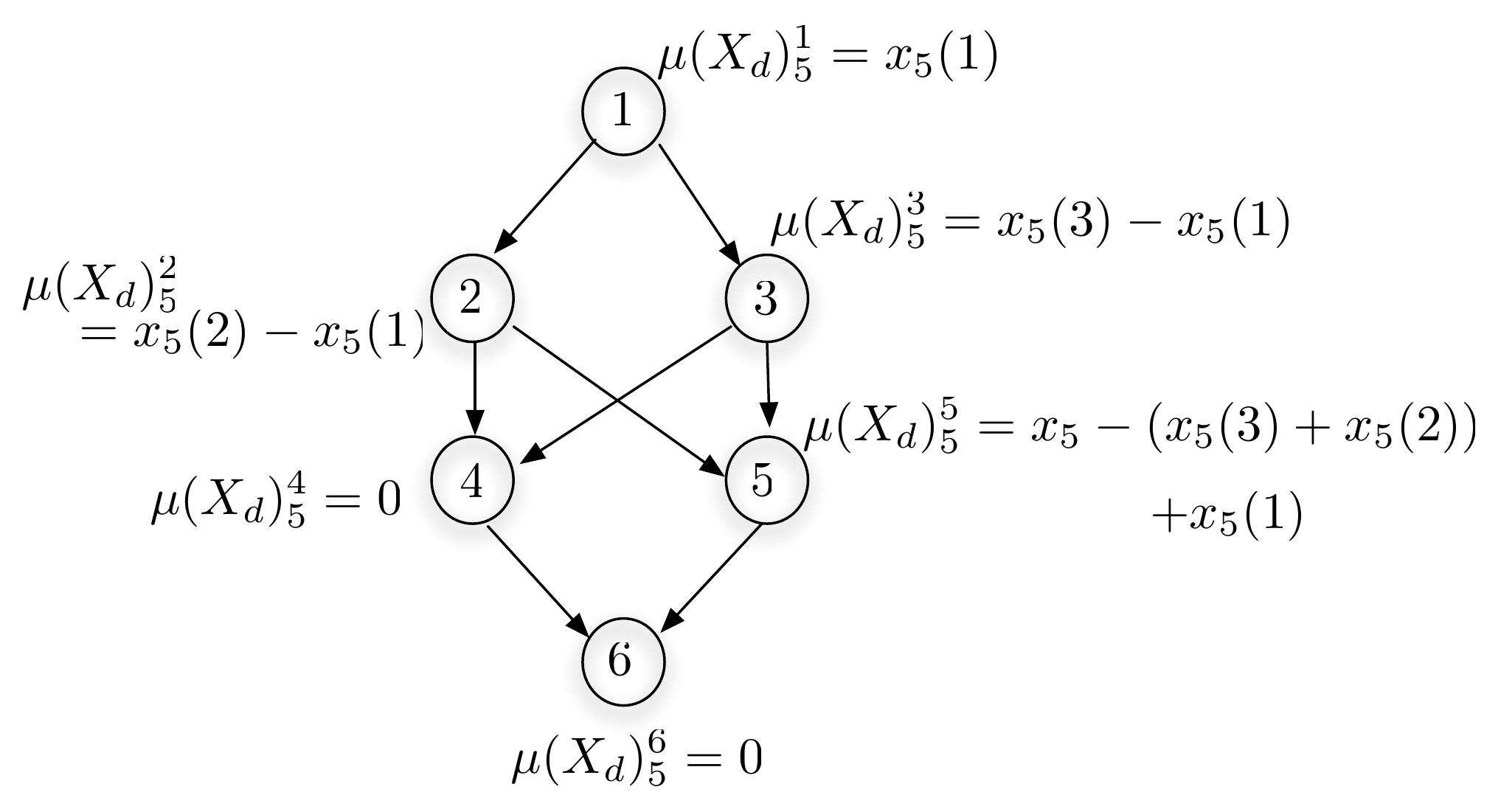}
  \end{center}
  \caption{Poset showing the differential improvement of the prediction of state $x_5$ at various subsystems.}
  \label{fig:diff_imp}
  \end{figure}
The prediction of $x_5$ at subsystem $1$ is $x_5(1)$ and the prediction of $x_5$ at subsystem $2$ is $x_5(2)$. The differential improvement in the prediction of $x_5$ at subsystem $2$ regarding the state $x_5$ is $x_5(2)-x_5(1)$. At subsystems $3$ the formula for the differential improvement is similar. The differential improvement in $x_5$ at subsystem $4$ is zero (since $4$ is offstream with respect to $5$). The differential improvement for $x_5$ at subsystem $6$ is zero (since $6$ is downstream with respect to $5$). These differential improvements are depicted in Fig. \ref{fig:diff_imp}. \emph{Capturing these differential improvements is precisely the role of the M\"{o}bius operator}.\end{example} 

We can now understand the role of \eqref{eq:X_{uo}1} better. Let $X$ be a local state and let $X_i$ denote the $i^{th}$ row of $X$. The $j^{th}$ entry of $X_i$ (i.e. $X^j_i$) corresponds to the prediction of $x_i$ at subsystem $j$, so that the row vector $X_i$ summarizes the predictions of the state $x_i$ at the different subsystems. 

Recalling \eqref{eq:X_{uo}1}, we have 
$$
X=\mu(X_d) \zeta^T.
$$
Restricting to the $i^{th}$ row and rewriting we have that
$$
X_i=\mu(X_d)_i \zeta^T = \sum _{k \preceq i} \mu(X_d)_i^k.
$$
The natural interpretation here is that in our architecture the \emph{``atoms'' for state prediction} are the differential improvements (or increments) $\mu(X_d)$. To compute the state prediction at subsystem $i$ one simply ``integrates'' (via $\zeta$) the increments corresponding to upstream subsystems $\up i$. This is depicted in Fig.~\ref{fig:increment}.

%
\begin{figure}[htbp]
  \begin{center}
  \includegraphics[scale=0.5]{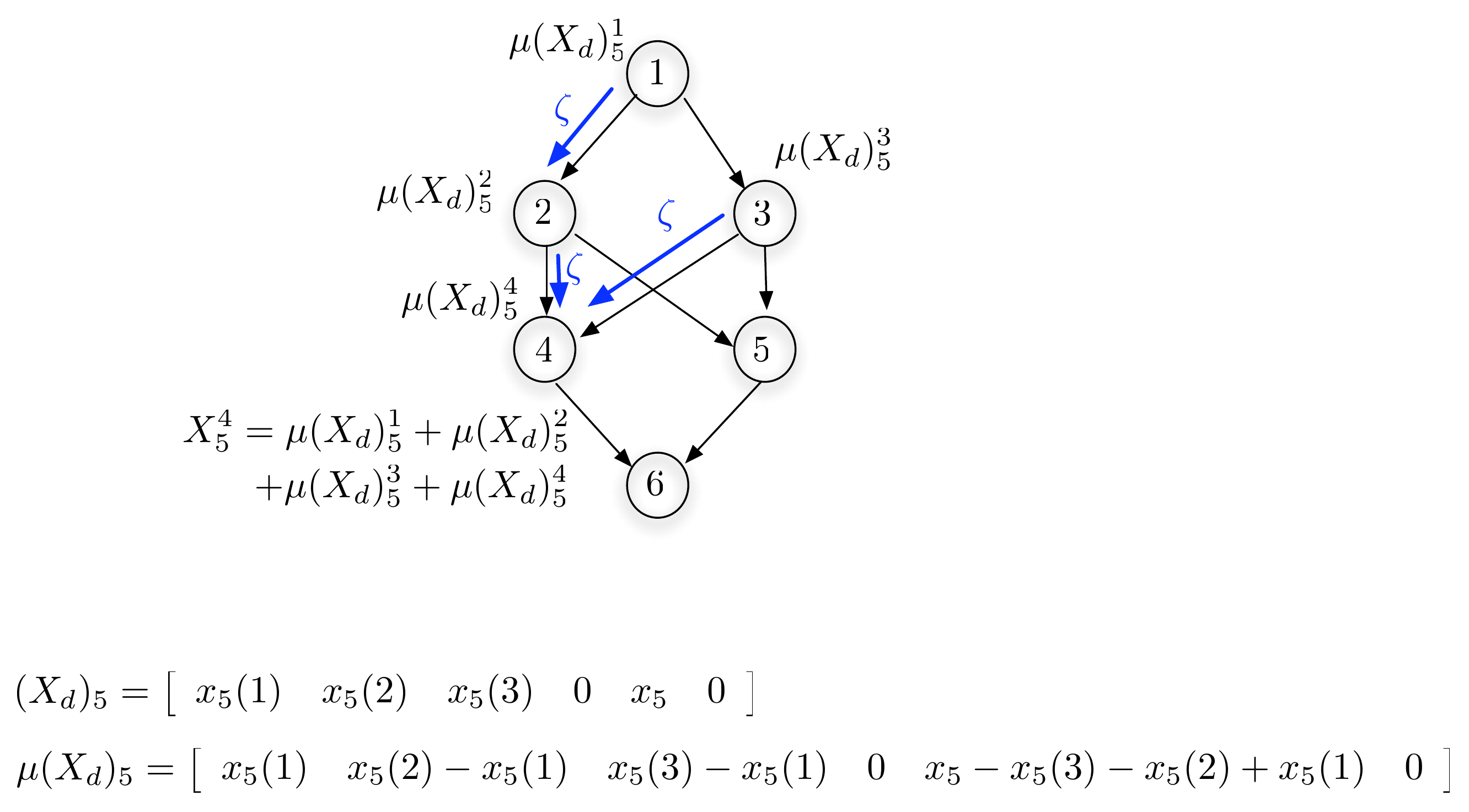}
  \end{center}
  \caption{This figure shows how the partial prediction of the state $x_5$ is computed at subystem $4$ using \eqref{eq:X_{uo}1}. States that are upstream of $4$ make predictions using free variables as reflected in the free variables $(X_d)_5$. The increments or differential improvements $\mu(X_d)_5$ will then form the ``atoms'' for prediction. To compute the prediction at subsystem $4$, one simply adds increments that are upstream with respect to $4$. This computation yields $X_5^4=\mu(X)_5^1+\mu(X)_5^2+\mu(X)_5^3+\mu(X)_5^4 = x_5(3)+x_5(2)-x_5(1)$. 
}
  \label{fig:increment}
  \end{figure}
  \begin{remark}
  As remarked earlier, our setup ensures that the increment $\mu(X_d)^i_j=0$ for $j \prec i$ or if $j$ and $i$ are not comparable. We briefly explain the intuition behind this. If $i$ is upstream of $j$, both subsystem $i$ and $j$ have access to the full state $x_j$ and hence there is no differential improvement in the state prediction. Furthermore, if $i$ and $j$ are unrelated, then $i$ carries no additional information about $x_j$ and again it is natural to have the differential improvement be zero.
  \end{remark}

Before we proceed we clarify that the entries of $X$ correspond to \emph{partial} predictions. We clarify the notion of a partial prediction with an example.
\begin{example}
Consider the system composed of three subsystems with $P=\left\{1,2,3 \right\}$ with $1 \preceq 3$ and $2 \preceq 3$ (see Fig. \ref{fig:partial_pred}):
$$
\left[ \begin{array}{c} \dot{x}_1  \\ \dot{x}_2\\ \dot{x}_3 \end{array} \right]= \left[ \begin{array}{ccc} A_{11} & 0 & 0 \\ 0 & A_{22} & 0 \\ A_{31} & A_{32} & A_{33} \end{array}\right]
\left[ \begin{array}{c} x_1  \\ x_2\\ x_3 \end{array} \right] + \left[ \begin{array}{ccc} B_{11} & 0 & 0 \\ 0 & B_{22} & 0 \\ B_{31} & B_{32} & B_{33} \end{array}\right]\left[ \begin{array}{c} u_1  \\ u_2\\ u_3 \end{array} \right].
$$
\begin{figure}[htbp]
  \begin{center}
  \includegraphics[scale=0.5]{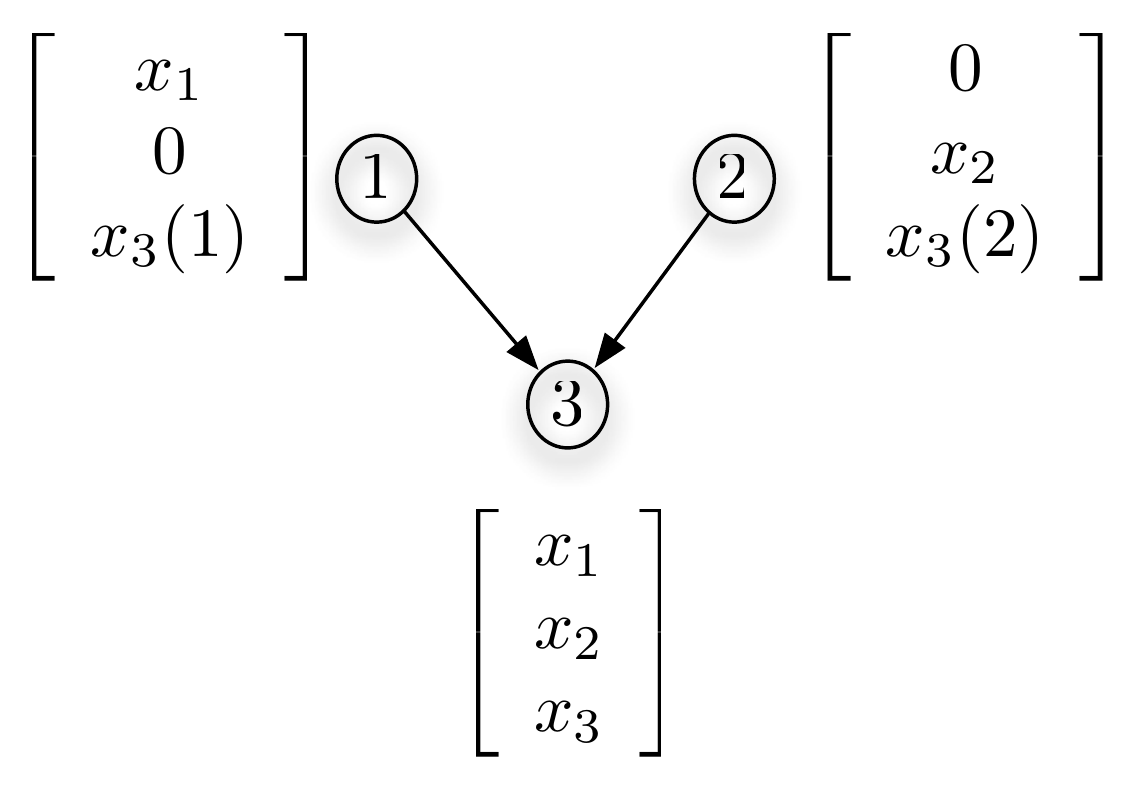}
  \end{center}
  \caption{Local state information at the different subsystems. The quantities $x_3(1)$ and $x_3(2)$ are partial state predictions.}
  \label{fig:partial_pred}
  \end{figure}

Note that subsystem $1$ has no information about the state of subsystem $2$. Moreover, the state $x_1$ or input $u_1$ does not affect the dynamics of $2$ (their respective dynamics are uncoupled). Hence the only sensible prediction of $x_2$ at subsystem $1$ is $x_2(1)=0$ (the situation for $u_2(1)$ is identical). However, both the states $x_1$, $x_2$ and inputs $u_1$, $u_2$ affect $x_3$ and $u_3$. Since $x_2$ and $u_2$ are unknown, the state $x_3(1)$ can at best be a \emph{partial} prediction of $x_3$ (i.e. $x_3(1)$ is the prediction of the \emph{component of} $x_3$ that is affected by subsystem $1$). Similarly $x_3(2)$ is only a partial prediction of $x_3$. Indeed, one can show that $x_3(1)+x_3(2)$ is a more accurate prediction of the state $x_3$, and when suitably designed, their sum converges to the true state $x_3$.
\end{example}

\subsection{Control Law}
We now formally propose the following control law:
\begin{equation} \label{eq:control_law}
U_d=\zeta(\mathbf{F} \circ \mu(X)).
\end{equation}
We make the following remarks about this control law.
\begin{remarks}
\begin{enumerate}
\item We note that \eqref{eq:control_law} specifies $U_d$ which amounts to specifying the input $(U_d)^i_i=u_i$ for all $i \in P$. It also specifies $(U_d)^i_j=u_j(i)$ for $i \prec j$ which is the prediction of the input $u_j$ at an downstream subsystem $i$.
\item 
Since $\hat{F}(i)$ is non-zero only on rows and columns in $\down i$, the controller 
respects the information constraints. Thus for any choice of gains $F(i)$, the resulting controller respects the information constraints. In this sense \eqref{eq:control_law} may be viewed as a parameterization of controllers.
\item The control law \eqref{eq:control_law} may be alternatively written as $U_d^i=\sum_{k \preceq i} F(k) \mu(X)^k$. The control law has the following interpretation. If $i$ is a minimal element of the poset $\mc{P}$, then $\mu(X)^i=X_d^i$, the vector of partial predictions of the state at $i$.  The local control law uses these partial predictions with the gain $F(i)$. If $i$ is a non-minimal element it aggregates all the control laws from $\up \up i$ and adds a \emph{correction term} based on the differential improvement in the global state-prediction $\mu(X)^i$. This correction term is precisely $F(i)\mu(X)^i$.
\end{enumerate}
\end{remarks}

\begin{example}
Consider a poset causal system where the underlying poset is shown in Fig \ref{fig:2.1}(d). 
The controller architecture described above is of the form $U_d^i=\sum_{k \prec i}F(k) \mu(X)^k$ (where $U^i$ is a vector containing the predictions of the global input at subsystem $i$). Noting that $(U_d)^i_i=u_i$, we write out the control law explicitly to obtain:
\begin{align*} \small
\ll \ba{c} u_1 \\ u_2 \\ u_3 \\ u_4\ea \rr=F(1)\ll \ba{c} x_1 \\ x_2(1) \\ x_3(1) \\ x_4(1) \ea \rr +F(2)\ll \ba{c} 0 \\ x_2-x_2(1) \\ 0 \\ x_4(2)-x_4(1) \ea \rr+
   F(3)\ll \ba{c} 0 \\ 0 \\ x_3-x_3(1) \\ x_4(3)-x_4(1) \ea \rr +F(4)\ll \ba{c} 0 \\ 0 \\ 0 \\ x_4-x_4(2)-x_4(3)+x_4(1) \ea \rr. 
\end{align*} \normalsize

\end{example}

\subsection{State Prediction}
Recall that at subsystem $i$ the states $x_j$ for $j \in \down \down
i$ are unavailable and must be predicted.  Typically, one would
predict those states via an observer. However, those states are
\emph{unobservable}; only the state $x_k$ for $k \in \up i$ are
observable, and are in fact directly available. In this situation,
rather than using an observer one constructs a \emph{predictor} to
predict the unobservable states. These predictions are computed by
the controller via prediction dynamics, which we now specify. 

Since the dynamics of the true state evolve according to $\dot{x}(t)=Ax(t)+Bu(t)+w(t)$, each subsystem can simulate these dynamics using the local states and inputs. Locally each subsystem implements the dynamics $\dot{X}^i(t)=AX^i(t)+BU^{i}(t)$. This can be compactly written as 
\begin{align}\label{eq:local_dyn}
\dot{X}(t)=AX(t)+BU(t).
\end{align}
We remind the reader that $X=X_d+X_{uo}$, and that $X_{uo}$ (consisting of upstream and offstream components) is determined from $X_d$ via \eqref{eq:X_{uo}1}. Consequently, the components in $X_{uo}$ are not free variables and one needs to check that \eqref{eq:local_dyn} constitutes a consistent set of differential equations. Projecting \eqref{eq:local_dyn} onto orthogonal components using $\Pi_d$ and $\Pi_{uo}$ we obtain
\begin{align*}
\dot{X}_d(t)=\Pi_d(AX(t)+BU(t)) \qquad \dot{X}_{uo}(t)=\Pi_{uo}(AX(t)+BU(t)).
\end{align*}
Before checking consistency, we simplify the differential equation for $X_d$ (the ``free variables''). The off-diagonal terms in $X_d$ correspond to the predictions of the downstream states, so that this is precisely the equation that governs the prediction or simulation component prescribed in Fig. \ref{fig:tf1}. Simplifying using $X=X_d+X_{uo}$ we get
\begin{equation}\label{eq:dynamics1}
\begin{split}
\dot{X}_d=AX_d(t)+BU_d(t)+R(t)\\
R(t)=\Pi_d(AX_{uo}(t)+BU_{uo}(t)).
\end{split}
\end{equation}
We think of $R(t)$ as the influence of the upstream components (and also the unrelated components) in predicting $X_d$. The dynamics \eqref{eq:dynamics1} correspond to the closed-loop dynamics. 
\begin{remark}
Equation \eqref{eq:dynamics1} along with \eqref{eq:control_law} formally specifies the controller. The controller states correspond to the off-diagonal entries of $X_d$ (i.e. the free variables of $X$). The number of states is equal to the number of intervals in the poset. 
\end{remark}

To check the consistency of \eqref{eq:local_dyn} we note that from \eqref{eq:X_{uo}1} we have that $\dot{X}_{uo}=\Pi_{uo}(\dot{\mu(X_d)}\zeta^T)$, whereas on the other hand from \eqref{eq:local_dyn} $\dot{X}_{uo}=\Pi_{uo}(AX+BU)$. It is sufficient to check that these two expressions are the same. To do so note the following chain of equalities:
\begin{align*}
\dot{X}_{uo}&=\Pi_{uo}(\dot{\mu(X_d)}\zeta^T) \\
&=\Pi_{uo}(\mu \left(AX_d+BU_d+R \right)\zeta^T) \qquad &\text{(from \eqref{eq:dynamics1})} \\
&=\Pi_{uo}( (A\mu(X_d)+B\mu(U_d))\zeta^T) \qquad &(\text{since }\mu(R)=0) \\
&=\Pi_{uo}( AX+BU) \qquad &(\text{using \eqref{eq:X_{uo}1}, \eqref{eq:U_{uo}}}). 
\end{align*}
\subsection{Separation Principle}
As a consequence of Lemma \ref{lemma:properties}, we see that $\mu(X)=\mu(X_d)$ and also $\mu(R)=A\mu(X_{uo})+B\mu(U_{uo})=0$.
Applying $\mu$ to \eqref{eq:dynamics1} we obtain the following \emph{modified closed-loop dynamics} in the new variables $\mu(X)$:
\begin{equation}\label{eq:dynamics2}
\dot{\mu(X)}(t)=A\mu(X)(t)+B\mu(U)(t).
\end{equation} 
Let us define 
$\mathbf{A}+\mathbf{B}\mathbf{F}=\left\{A+B\hat{F}(1), \ldots, A+B\hat{F}(s) \right\}.$
From \eqref{eq:control_law}, and the fact that $\mu(\zeta(\mathbf{F} \circ \mu(X)))=\mathbf{F} \circ \mu(X)$ we will momentarily see that the modified closed-loop dynamics are:
 \begin{equation} \label{eq:sep}
 \dot{\mu(X)}(t)=(\mathbf{A}+\mathbf{B}\mathbf{F}) \circ \mu(X)(t).
 \end{equation}
These dynamics describe how the differential improvements in the state evolve. If one picks $U$ such that $\mu(U)$ stabilizes $\mu(X)$, the differential improvements are all stabilized. Thus $\mu(X)$ converges to zero, the state predictions become accurate asymptotically and the closed-loop is also stabilized. We show that \eqref{eq:control_law} achieves this with an appropriate choice of local gains.
\begin{theorem} \label{theorem:separation_principle}
Let $F(i)$ be chosen such that $A(\down i, \down i)+ B(\down i, \down i)F(i)$ is stable for all $i \in P$. Then the control law \eqref{eq:control_law} with local gains $F(i)$ renders \eqref{eq:dynamics2} stable.
\end{theorem}
\begin{proof}
Since $U_d=\zeta(\mathbf{F}\circ \mu(X))$ it follows that 
\begin{align*}
\mu(U_d)=\mu(U)=\mu \left( \zeta(\mathbf{F}\circ \mu(X)) \right) =\mathbf{F}\circ \mu(X). 
\end{align*}
The last equality follows from Lemma \ref{lemma:properties} and the fact that $\mathbf{F} \circ \mu(X) \in \inc$. As a consequence, $\mu(U)^i=\hat{F}(i) \mu(X)^i$ for all $i \in P$.
Hence the closed-loop dynamics \eqref{eq:dynamics2} become:
$$
\dot{\mu(X)}^i(t)=\left(A+B\hat{F}(i)\right)\mu(X)^i(t).
$$


Recalling that $\mu(X)$ is a local variable so that $\mu(X)^i$ (viewed as a vector) is non-zero only on $\down i$ it is easy to see that
these dynamics are stabilized exactly when $F(i)$ are picked such that  $A(\down i, \down i)+ B(\down i, \down i)F(i)$ are stable.

 \end{proof}

The dynamics of the different subsystems $\mu(X)^i$ are \emph{decoupled}, so that the gains ${G}(i)$ may be picked independent of each other. This may be viewed as a \emph{separation principle}. Henceforth, we will assume that the gains $F(i)$ have been picked in this manner.
Since the closed loop dynamics of the states $x_{i}(j)$ are related by an invertible transformation (i.e. $X_d=\zeta(\mu(X_d))$), if the modified closed-loop dynamics \eqref{eq:sep} are stable, so are the closed-loop dynamics \eqref{eq:dynamics1}.
\subsection{Controller Realization}
 We now describe two explicit controller realizations. The natural controller realization arises from the closed-loop dynamics \eqref{eq:dynamics1} along with the control law \eqref{eq:control_law} to give:
 \begin{align*}
 \dot{\mu(X)}(t)&=AX_d(t)+BU_d(t)+R(t) \\
 U_d(t)&=\zeta(\mathbf{F}\circ \mu(X_d))(t).
 \end{align*}

While the above corresponds to a natural description of the controller, it is possible to specify an alternate realization. This is motivated from the following observation. The control input $U$ depends only on $\mu(X)$. Hence, rather than implementing controller states that track the state predictions $X$, it is natural to implement controller states that compute $\mu(X)$ directly. Hence an equivalent realization of the controller is:
\begin{equation} \label{eq:controller}
\begin{split}
\dot{\mu(X)}(t)&=A\mu(X)(t)+B\mu(U)(t)\\
 U_d(t)&=\zeta(\mathbf{F}\circ \mu(X))(t).
\end{split}
\end{equation}

\subsection{Structure of the Optimal Controller}

Consider again the poset-causal system considered in \eqref{eq:1}. 
Recall that the system \eqref{eq:2} may be viewed as a map from the inputs $w, u$ to outputs $z,x$ via
\begin{align*}
z&=P_{11}w+P_{12}u \\
x&=P_{21}w+P_{22}u
\end{align*}
where \small
\begin{equation} \label{eq:1}
\ba{rl}
\ll
\ba{cc} P_{11} & P_{12} \\
P_{21} & P_{22}
\ea \rr 
=
\ll \ba{c|cc} 
A&I&B \\ \hline C & 0 & D \\ I & 0 & 0
\ea \rr.
\ea
\end{equation}
\normalsize
(We refer the reader to \cite{Zhou} as a reminder of standard LFT notation used above).
In this paper we will assume that $A \in \mc{I}(\mc{P})$ and   $B \in \mc{I}(\mc{P})$. Indeed, this assumption ensures that the plant $P_{22}(z)=(zI-A)^{-1}B \in \mc{I}(\mc{P})$. 

Consider the optimal control problem:
\begin{equation} \label{eq:4}
\begin{aligned}
& \underset{K}{\text{minimize}} & &  \|P_{11}+P_{12}K(I-P_{22}K)^{-1}P_{21} \|^{2}  \\
& \text{subject to} & & K \text{ \ stabilizes \ } P, \; K \in \mc{I}(\mc{P}).
\end{aligned}
\end{equation}
The solution $K^{*}$ is the $\mc{H}_2$-optimal controller that obeys the poset-causality information constraints described in Section~\ref{sec:2}. The solution to this optimization problem was presented in \cite[Theorem 3]{poset_h2_journal}. The main idea behind the solution procedure is as follows. Using the fact that $P_{21}, P_{22} \in \inc$ are square and invertible (due to the availability of state feedback) it is possible to reparametrize the above problem via $Q=K(I-P_{22}K)^{-1}P_{21}$. Indeed, this relationship is invertible and the incidence algebra structure ensures that $Q \in \inc$ if and only if $K \in \inc$. Using this the above optimization problem may be rewritten as:
\begin{equation} 
\begin{aligned}
& \underset{Q}{\text{minimize}} & &  \|P_{11}+P_{12}Q \|^{2}  \\
& \text{subject to} & &  Q \in \mc{I}(\mc{P}).
\end{aligned}
\end{equation}
Using the fact that the $\mc{H}_2$ norm is column-separable, it is possible to decouple this optimization problem into a set of $s$ optimization problems. Each optimization problem involves the solution to a standard Riccati equation. The solution to each yields the columns of $Q^{*} \in \inc$, from which the optimal controller $K^{*} \in \inc$ may be recovered. An explicit formula for the optimal controller and other details may be found in \cite{pari_thesis,poset_h2}.

In \cite{poset_h2_journal}, we obtain matrices $K(\down j, \down j)$ by solving a system of decoupled Riccati equations via $(K(\down j, \down j), Q(j), P(j))=\text{Ric}(A(\down j, \down j), B(\down j, \down j), C(\down j), D(\down j))$ (we use slightly different notation and reversed conventions in that paper, see \cite{poset_h2_journal} for details). The optimal solution $K^*$ to \eqref{eq:4} is related to the proposed architecture as described below in Theorem \ref{theorem:struct}. Its proof is provided in the appendix.

 \begin{theorem} \label{theorem:struct}
 The controller \eqref{eq:controller} with gains $F(i)=K(\down i, \down i)$ for all $i \in P$ is the optimal solution to the control problem \eqref{eq:4}.
 \end{theorem}
Theorem \ref{theorem:struct} establishes that the controller architecture proposed in this paper is also \emph{optimal} in the sense of the $\mc{H}_2$ norm.

\section{A Block-Diagram Interpretation}\label{sec:5}
We now interpret the separation principle described in Theorem \ref{theorem:separation_principle} from a block-diagram perspective.
We introduce some additional notation to be used in this section. If $Y$ is a matrix, we denote by $\vec(Y)$ to be its standard vectorization via column concatenation. Given matrices $M$ and $N$, $M \kron N$ denotes their Kronecker product, and we recall the standard identity:
$$
\vec(NXM^T)=(M \kron N) \vec(X).
$$

Recall the definition of $\Pi_d$ in Definition \ref{def:matrix_mult} as being the projection of a matrix $X$ onto the incidence algebra. We will abuse notation and also define $\Pi_d$ to be the linear operator that acts on the vectorization $\vec(X)$ and zeroes components whose indices do not belong to the incidence algebra (the usage will be clear from the context).

The elementary blocks that appear in our block-diagram representation are the following:
\begin{itemize}
\item The plant $G$, which maps the inputs $u$ to the states $x$,
\item The transfer functions which play the role of predicting the local state variables $X^i$ from the states $x_j$ and inputs $u_j$ for $j \in \down i$ via \eqref{eq:dynamics1}. We call all these transfer functions collectively the ``simulator'', because their role may be interpreted as that of simulating upstream states,
\item The map $\bar{\mu}$ which takes as input $\vec(X)$ and computes $\vec(\mu(X))$,
\item The local gains $F(1), \ldots F(s)$,
\item The map $\bar{\zeta}$ which takes as input $\vec(\mu(U))$ and computes $\vec(U_d)$.
\end{itemize}
In the closed-loop system all these transfer functions are interconnected as shown in Fig. \ref{fig:tf1}.

We define the opertators $\bar{\zeta}$ and $\bar{\mu}$ as follows:
\begin{align} \label{eq:mu_vec}
\bar{\zeta}:=\Pi_d(\zeta \kron I) \qquad \bar{\mu}:=\Pi_d(\mu \kron I).
\end{align} 
Since $X_d=\zeta(\mu(U))$, note that $\vec(U_d)=\bar{\zeta} \vec(\mu(U))$ (this clarifies the role of $\bar{\zeta}$ inFig. \ref{fig:tf1}). Similarly, $\vec(\mu(X))=\bar{\mu}\vec(X)$. 
As one might expect, $\bar{\zeta}$ and $\bar{\mu}$ are invertible restricted to $\inc$.
\begin{lemma}\label{lemma:inverses_vec}
For $\bar{\zeta}$ and $\bar{\mu}$ as defined in \eqref{eq:mu_vec}
$$
\bar{\zeta}\bar{\mu}=\bar{\mu}\bar{\zeta}=\Pi_d.
$$
\end{lemma}
\begin{proof}
This follows simply by vectorizing the first and second parts of Lemma \ref{lemma:properties}.
\end{proof}

Figure \ref{fig:tf1} explains the closed-loop at block diagram level. The role of the simulator is to form predictions of downstream states from the available states and inputs. These predictions $\vec(X)$ are then processed by $\bar{\mu}$, to produce $\vec(\mu(X))$. These are then composed with the local gains to produce $\vec(\mu(U))$. Finally, $\bar{\zeta}$ acts on $\vec(\mu(U))$ to produce $\vec(U_d)$. Internally within the simulator one can map the downstream components $U_d$ to the local input $U$ (using \eqref{eq:U_{uo}}) via
$$
\Theta=(\zeta \kron I)\Pi_d(\mu \kron I).
$$ 
It may be easily verified using \eqref{eq:U_{uo}} that $\Theta \vec(U_d)=\vec(U)$. The overall composition of the plant $G$ and the simulator can be combined to give a transfer function which we denote by $G_{\vec}$ as shown in Fig \ref{fig:tf2}. It has an explicit formula given by:
\begin{align} \label{eq:Gvec}
G_{\vec}=(I \kron G)(\zeta \kron I)\Pi_d(\mu \kron I).
\end{align}
Thus $G_\vec$ is the overall tranfer function that maps $\vec(U_d)$ (the true and predicted downstream inputs) to $\vec(X)$ (the local states). We illustrate this with an example. \vspace{-7pt}
\begin{example}
For the poset in Fig. \ref{fig:2.1}(a), we have \small
\begin{align*}
\vec(X_d)=\ll \ba{c}x_1 \\ x_2(1) \\ 0 \\ x_2 \ea \rr & \qquad
\vec(U_d)=\ll \ba{c}u_1 \\ u_2(1) \\ 0 \\  u_2 \ea \rr.
\end{align*} \normalsize
Furthermore, the map $G$ is given by 
$
G=\ll \ba{cc}
G_{11} & 0 \\ G_{21} & G_{22}
\ea \rr, 
$
and the map $G_{\vec}$ as defined in \eqref{eq:Gvec} is given by: \small
$$
G_{\vec}=\ll \ba{cccc} 
G_{11} & 0 & 0 & 0\\
G_{21} & G_{22} & 0 \\
0 & 0 & 0 & 0 \\
G_{21} & 0 & 0 & G_{22}
\ea \rr.
$$ \normalsize
For this poset the matrices $\bar{\zeta}$ and $\bar{\mu}$ are given by: \small
\begin{align*}
\bar{\zeta}= \ll \ba{cccc} I & 0 & 0 & 0 \\ 0 & I &0 & 0\\ 0 & 0 & 0 & 0\\ 0 & I & 0 & I   \ea \rr & \qquad
\bar{\mu}= \ll \ba{cccc} I & 0 & 0 & 0 \\ 0 & I &0 & 0\\ 0 & 0 & 0 & 0\\ 0 & -I & 0 & I \ea \rr. 
\end{align*} \normalsize
It is straightforward to verify that $\vec(X)=G_\vec \vec(U)$. The following important identity may be verified for this example:
$$
\bar{\mu}G_{\vec}\bar{\zeta}=\ll \ba{cccc} 
G_{11} & 0 & 0 & 0 \\
G_{21} & G_{22} & 0 & 0 \\
0 & 0 & 0 & 0 \\
0 & 0 & 0 & G_{22}
\ea \rr,
$$
which is a block diagonal matrix.
\end{example}
As indicated in Fig. \ref{fig:tf1}, the collective map from $\vec(U_d)$ to $\vec(X)$ (which collects the plant $G$, $\Theta$, and the simulation block into a single transfer function) is simply given by $G_{\vec}$. Thus the block-diagram in Fig. \ref{fig:tf1} can be simplified to Fig. \ref{fig:tf2}. 
\begin{figure}[htbp]
  \begin{center}
  \includegraphics[scale=0.4]{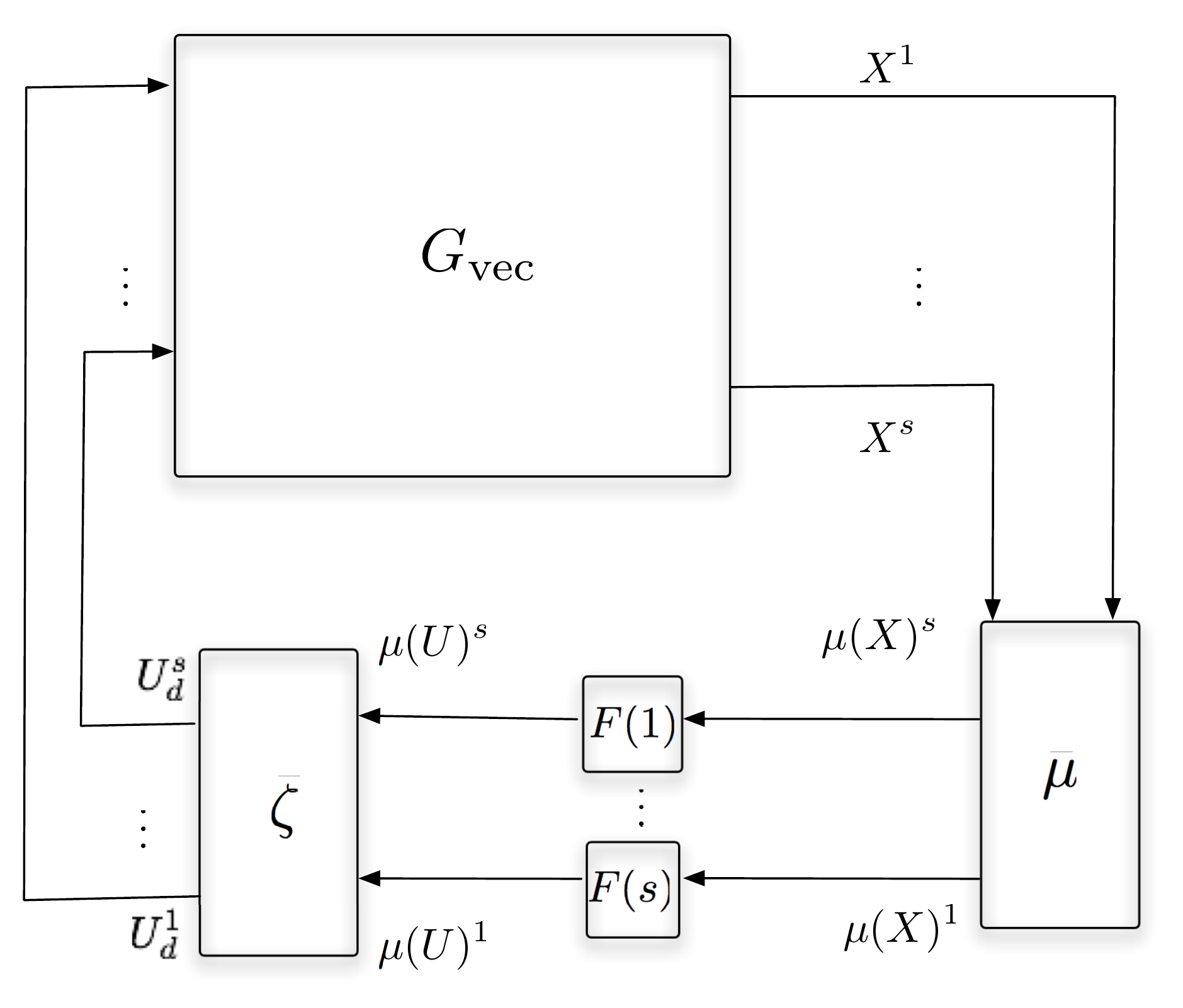}
  \end{center}
  \caption{A simplified block-diagram representation of the control architecture.}
  \label{fig:tf2}
  \end{figure}
The matrix $G_{\vec}$ satisfies the following: 
\begin{theorem}
The matrix
$\bar{\mu}G_{\vec} \bar{\zeta}$ 
is block diagonal.
\end{theorem}
\begin{proof}
Using \eqref{eq:Gvec}, we have \vspace{-15pt}
\begin{align*}
\bar{\mu}G_{\vec}\bar{\zeta}&=\bar{\mu}(I \kron G)(\zeta \kron I)\Pi_d(\mu \kron I)\bar{\zeta} \\
&=\Pi_d(\mu \kron I)(I \kron G)(\zeta \kron I)\Pi_d(\mu \kron I)\Pi_d(\zeta \kron I) \\
&=\Pi_d(\mu \kron I)(\zeta \kron I)(I \kron G)\Pi_d(\mu \kron I)\Pi_d(\zeta \kron I)  & \text{(using } (A \kron B) (C \kron D)=AC \kron BD) \\
&=\Pi_d (I \kron G)\Pi_d(\mu \kron I)\Pi_d(\zeta \kron I)  &\text{(using } (\mu \kron I) (\zeta \kron I)=I) \\
&=\Pi_d (I \kron G)\Pi_d & \text{(using Lemma \ref{lemma:inverses_vec})}.
\end{align*}
The last expression is clearly block-diagonal.
\end{proof}
In terms of this block-diagram approach, the role of $\bar{\mu}$ and $\bar{\zeta}$ become very transparent: it is simply to diagonalize the map $G_\vec$. Once this diagonalization occurs, the controller simply applies a set of diagonal gains to stabilize the closed-loop. This also illustrates the separation principle at the block-diagram level. As mentioned in the preceding discussion, the architecture illustrated in block-diagram Fig \ref{fig:tf2} is also optimal, in that appropriate choice of the gains $F(i)$ yield optimal controllers.
%

\section{Connections to the Youla parameterization} \label{sec:6}
The Youla parameterization (and the related work on purified output feedback \cite{purified1,Goulart}) is intimately related to M\"{o}bius inversion. We examine their relationship in this section. We begin with a brief review of the Youla parameterization. We will examine the relationship in a discrete-time setting as this will make the presentation much simpler. Consider the system:
\begin{equation} \label{eq:disc_time}
\begin{split}
x[t]&=Ax[t-1]+Bu[t-1]+w[t-1] \\
y[t]&=x[t].
\end{split}
\end{equation}
For simplicity, we will assume that $A$ is stable (this can be easily achieved by choosing a static $K$ which is diagonal by picking $K_{ii}$ such that $A_{ii}+B_{ii}K_{ii}$ is stable, without affecting the set of achievable closed-loop maps). The Youla parametrization \cite[pp. 221-231]{Zhou} exploits the observation that while the set of achievable closed-loop maps is \emph{linear-fractional} with respect to the controller variable, it is \emph{affine} in terms of the parameter $Q:=K(I-P_{22}K)^{-1}$.

In the case of state-feedback with a stable plant, the Youla parameterization reduces to a particularly simple form, which we now describe. At time $t$, using the information of the state $x[t-1]$, the controller implements a simulation $\hat{P}$ of the plant to compute a prediction of the state at time $t$ via:
\begin{align}
\hat{x}[t]=Ax[t-1]+Bu[t-1].
\end{align}
Note that the simulation does not have access to the disturbances $w$ and hence simply sets it to zero. It then uses the output of the plant $x[t]$ (state-feedback) and then computes the difference $x[t]-\hat{x}[t]=w[t-1]$. It then processes $w[t-1]$ using an arbitrary causal filter $Q$ and sets the input $u=Qw$.
This can be summarized using a block diagram as shown in Fig. \ref{fig:youla1}(a).

\begin{figure}[htbp]
  \begin{center}
  \includegraphics[scale=0.35]{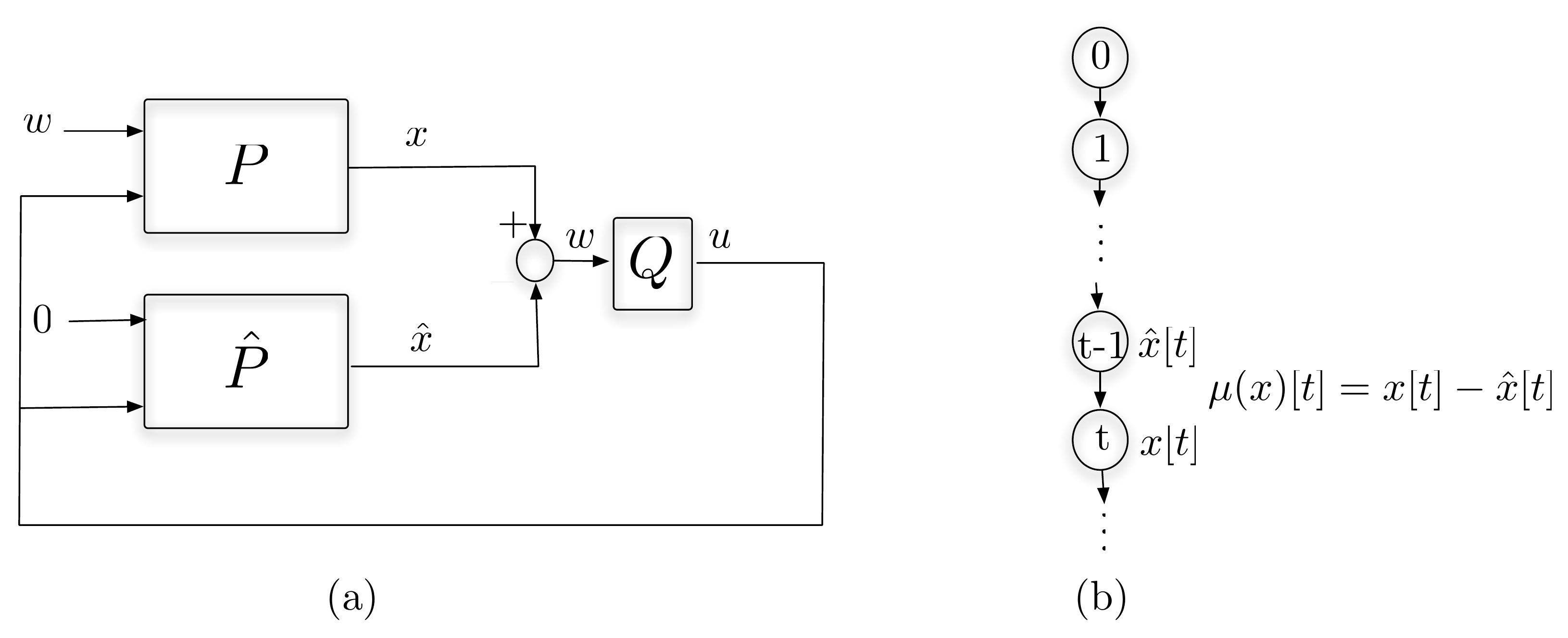}
  \end{center}
\vspace{-15pt}
  \caption{(a) A block diagram interpretation of the Youla parameterization. Here it is assumed that the plant $P$ is stable and state-feedback is available. The simulator $\hat{P}$ predicts the state one time step in the future, using which the disturbance $w$ is reconstructed. The filter $Q$ then implements a disturbance feed-forward policy. (b) The time axis viewed as a poset $\mc{T}$. At time $t-1$ the simulator predicts the state $x$ downstream at time $t$. M\"{o}bius inversion reconstructs the disturbance $w[t-1]$.}
  \label{fig:youla1}
\end{figure}

We point out that the policy $u=Qw$ is a \emph{disturbance feed-forward policy}. It has been observed in various papers in the literature that while state-feedback policies may yield a complicated dependence between the closed-loop map and the controller variable, an equivalent reparametrization using disturbance feed-forward policies yields an affine dependence. For example, this observation was made in the context of robust optimization for linear systems in the framework of ``purified output-feedback'' \cite{purified1} and also in \cite[Chapter 6]{gattami_thesis}.

The key step in reformulating a feedback problem into a disturbance feed-forward problem is the explicit reconstruction of the disturbance $w$ using the output of the plant and the simulator. This computation may be naturally viewed as a M\"{o}bius inversion operation. To the dynamic evolution of the discrete-time system \eqref{eq:disc_time} one can naturally associate the poset $\mc{T}=(\mathbb{N}, \preceq)$, i.e. the time-axis indexed by the natural numbers (equipped with the standard ordering). This poset is simply the linear poset indexed by the integers (see Fig. \ref{fig:youla1}(b)), and the system-theoretic notion of causality is simply the specialization of our notion of poset-causality specialized to this poset.

%

Consider the variable $x[t]$ (i.e. the state of the system \eqref{eq:disc_time}) as a function on this poset. For elements $k$ such that $t \preceq k$, $x[t]$ is available and for elements $k$ such that $k \prec t$, it is unavailable. Indeed at the element $t-1$ a prediction of $x[t]$ may be computed via $\hat{x}[t]=Ax[t-1]+Bu[t-1]$ (this is precisely the role of the simulator $\hat{P}$ described above). Using the M\"{o}bius inversion formula for the linear poset $\mc{T}$ we have
$
\mu(x[t])=x[t]-\hat{x}[t]=w[t-1].
$
Hence the disturbance computation may be viewed simply as a M\"{o}bius inversion on $\mc{T}$.

It is possible to extend this interpretation to poset-causal systems with multiple  subsystems. In this case, the system of interest is of the form \eqref{eq:disc_time}, where $A, B \in \inc$. The poset associated to the dynamic evolution of this system is then the \emph{product poset} $\mc{T} \times \mc{P}$. An example of product of this is shown in Fig. 
\vspace{-15pt}
\ref{fig:youla3}(a).
\begin{figure}[htbp]
  \begin{center}
  \includegraphics[scale=0.5]{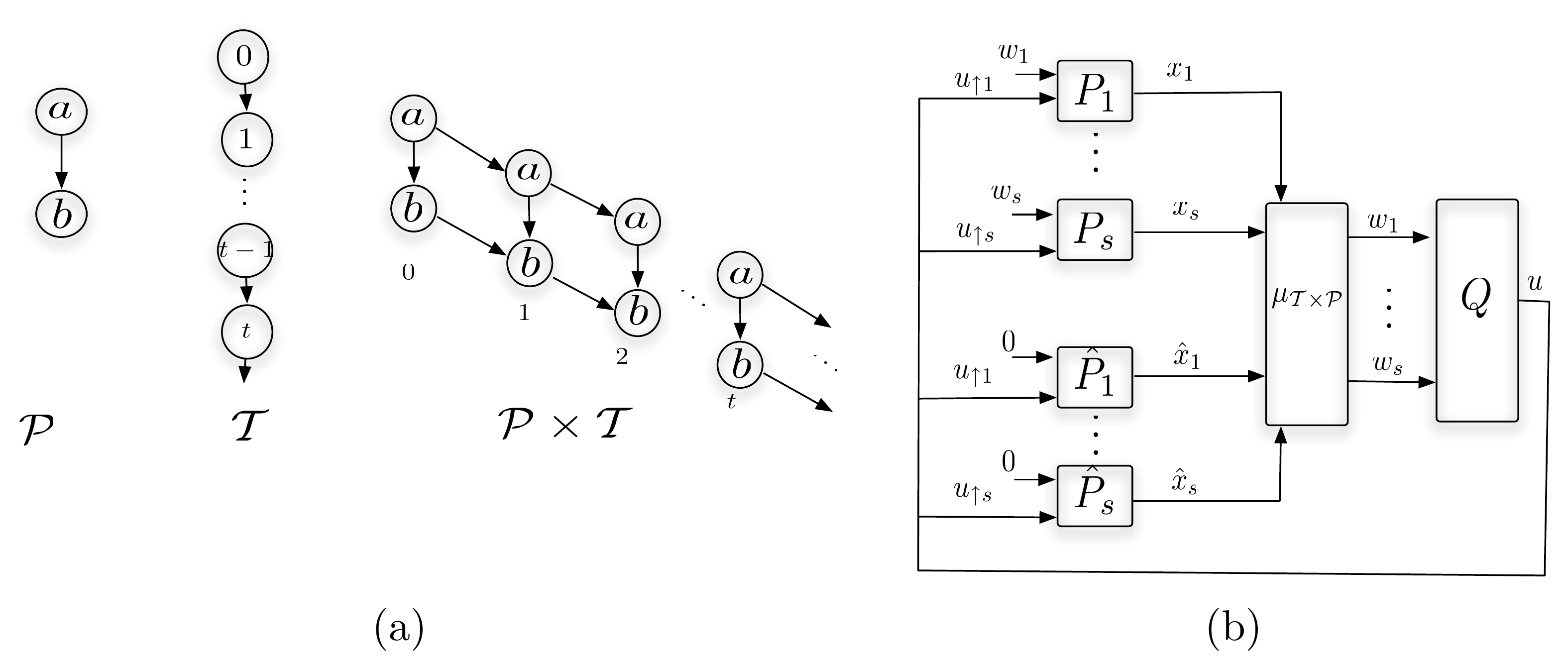}
  \end{center}
\vspace{-15pt}
  \caption{(a) The poset $\mc{P}$ captures causality between the subsystems and the poset $\mc{T}$ captures causality with respect to time. Their product poset is shown on the right. (b) The Youla parameterization implements simulation, followed by M\"{o}bius inversion to compute the disturbances. It then implements a disturbance feedforward policy via a causal (with respect to $\mc{P}\times \mc{T}$) filter $Q$.}
  \label{fig:youla3}
\end{figure}
As explained in earlier sections, subsystems maintain local states which are summarized by the local variable $X$. Using information available at time $t-1$, the prediction of $X[t]$ may be computed as:
\begin{align} \label{eq:disc_local_pred}
\hat{X}[t]=AX[t-1]+BU[t-1].
\end{align}
However, due to the disturbance, the value of $X[t]$ is given by
\begin{align} \label{eq:disc_local}
X[t]=AX[t-1]+BU[t-1]+W[t-1]\zeta^{T},
\end{align}
where $W=\diag(w)$.
The local state variable $X[t]$ may be viewed as a function on the poset $\mc{T} \times \mc{P}$, and hence one may define its M\"{o}bius inverse with respect to this poset. As a consequence of the product structure of the underlying poset, we have the following important lemma.
\begin{lemma}
Suppose $X[t]$ satisfies \eqref{eq:disc_local}. 
Let $\mu_{\mc{T} \times \mc{P}}$ denote the M\"{o}bius operator of the poset $\mc{T} \times \mc{P}$. Then
$$
\mu_{\mc{T} \times \mc{P}}(X[t])=W[t-1].
$$
\end{lemma}\vspace{-15pt}
\begin{proof}
We will use $\mu(X)$ to denote the M\"{o}bius inversion of the local variable $X$ with respect to the poset $\mc{P}$ as defined in Definition \ref{def:mobius_op}.
It is well-known \cite[Proposition 5]{rota} that the M\"{o}bius operator factorizes for product posets as:
\begin{align*}
\mu_{\mc{T} \times \mc{P}}X[t]&=\mu_{\mc{T}} \mu_{\mc{P}}X[t] 
=\mu(X)[t]-\mu(\hat{X})[t] 
=\mu(W[t-1] \zeta^{T}) 
=W[t-1].
\end{align*}
\end{proof}
Thus the role of the M\"{o}bius operator is to compute the local disturbance $w_i$. Once these disturbances are computed, one may process it using a filter $Q$ that is causal with respect to $\mc{T} \times \mc{P}$ to obtain the input $u$. This is depicted in Fig. \ref{fig:youla3}(b).
\section{Conclusions}
In this paper we considered the problem of designing decentralized
poset-causal controllers for poset-causal systems. We studied the
architectural aspects of controller design, addressing issues such as
the role of the controller states, and how the structure of the poset
should affect the architecture. We proposed a novel architecture in
which the role of the controller states was to locally predict the
unknown ``downstream'' states. Within this architecture the controller
itself performs certain natural local operations on the known and
predicted states. These natural operations are the well-known zeta and
M\"{o}bius operations on posets.

Having proposed an architecture, we proved two of its important
structural properties. The first was a \emph{separation principle}
that enabled a decoupled choice of gains for each of the local
subsystems. The second was establishing the \emph{optimality
  properties} of this architecture with respect to the
$\mc{H}_2$-optimal decentralized control problem. The proposed
M\"obius-based architecture is quite natural, has very appealing
interpretations, and can be easily extended to more complicated. These extensions will be the subject of future
work.

%

\bibliography{architecture}

\appendix
The required optimal solution was derived in \cite{poset_h2_journal}[Theorem 3] and we first develop some notation and concepts that have been used therein.
Let us define $q(i)=\mu(X)^i_{\down \down i}$ (so that the vector $q(i)$ has only those components of $\mu(X)^i_j$ such that $i \prec j$), and $q_j(i)=\mu(X)^i_j$.  Note also that $\mu(X)^j_j=x_j-\sum_{k \prec j}q_j(k)$ from \eqref{eq:mobius2}. Let us define $A^{cl}(j)=A(\down j, \down j) + B(\down j, \down j)F(j)$. The closed-loop dynamics \eqref{eq:sep} at subsystem $j$ reduce to:
\small
\begin{equation}  \label{eq:sep2} \begin{split}
\ll \ba{c} \dot{x}_j-\sum_{k \prec j}\dot{q}_j(k) \\ \dot{q}(j) \ea \rr = 
\ll \ba{cc} A^{cl}_{11}(j) & A^{cl}_{12}(j) \\ A^{cl}_{21}(j) & A^{cl}_{22}(j) \ea \rr \ll \ba{c} x_j-\sum_{k \prec j}q_j(k) \\ q(j) \ea \rr(t).
\end{split} \end{equation}\normalsize \vspace{-10pt}
Note that from \eqref{eq:control_law} (keeping in mind that $\mu(X)^i_j=0$ if $j \npreceq i$, and that $u_j=U^j_j$), the control law assumes the form: \small
\begin{align*}
u_j&=\sum_{k \preceq j}\hat{F}^{(j)}(k)\mu(X)^k =\sum_{k \preceq j} F^{(j)}(k) \ll \ba{c} x_k-\sum_{l \prec k}q_k(l) \\ q(k) \ea \rr(t).
\end{align*} \normalsize
(Recall that $F^{(j)}(k)$ is the $j^{th}$ row of the matrix $F(k)$).
Thus the subsystems need to compute $q(j) = \mu(X)_{\down \down j}$, (the differential improvements in the state predictions at subsystem $j$) to implement the control law. This is an important feature of the control law: \emph{the controller states correspond to the differential improvements $\mu(X)$ rather than the state $X$ itself}.

It may be verified from \eqref{eq:sep2} and \eqref{eq:control_law} that the explicit controller for subsystem $j \in P$ assumes the following form: \small
\begin{equation} \label{eq:controller2}
\begin{split}
\dot{q}(j) &=A^{cl}_{22}(j)q(j)+A^{cl}_{21}(j)\left(x_j-\sum_{k \prec j}q_j(k)\right) \\
u_j(t)&=\sum_{k \preceq j} F^{(j)}(k) \ll \ba{c} x_k-\sum_{l \prec k}q_k(l) \\ q(k) \ea \rr(t).
\end{split}
\end{equation} \normalsize
Furthermore, note that at subsystem $j$, $\mu(X)_{ij}=0$ for $j \npreceq i$. Hence, only the states $\mu(X)_{ij}$ for $i \in \down \down j$ need to be computed. Let  $\mu(X)_{\down \down j}=\ll \mu(X)_{ij} \rr_{i \in \down \down j}$.
\begin{proof}[Proof of Theorem \ref{theorem:struct}]
As mentioned, the optimal controller is given by \cite[Theorem 3]{poset_h2_journal}. We will show that when we pick $F(i)=K(\down i, \down i)$ in \eqref{eq:controller}, we recover this controller. In \cite{poset_h2_journal}, the matrices $\mathbf{A}$, $A_{\phi}$, $B_{\phi}$, $C_{\phi}$, $\Pi_1$, and $\Pi_2$ were defined. It is straightforward to verify that:  
\begin{align*} 
\text{diag}(A^{cl}(j))&=\mathbf{A}  &\qquad  
\text{diag}(A_{22}^{cl}(j))&=A_{\Phi} \\
\text{diag}(A_{21}^{cl}(j))&=B_{\Phi} & \qquad   
\left[\sum_{k \prec j} q_j(k) \right]_{j\in P}&=C_\Phi q.  \label{eq:q_identity}
\end{align*}
Letting $q$ be the vectorization of $q(j)$ for $j \in P$ via $q=[q(j)]_{j \in P}$, we may rewrite the dynamics in \eqref{eq:controller2} as
\begin{equation} \label{eq:dynamics_rewritten}
\dot{q}=A_{\Phi}q(t)+B_{\Phi}(x(t)-C_\Phi q(t)).
\end{equation}
Further, note that the vectorization of the control law equation in \eqref{eq:controller2} yields \small
$$
u= \left[ \sum_{k \preceq j} K^{(j)}(\up k, \up k) \ll \ba{c} x_k-\sum_{l \prec k}q_k(l) \\ q(k) \ea \right] \right]_{j \in P}=C_Q\Pi_2 q+C_Q \Pi_1(x-C_{\Phi}q).
$$ \normalsize
Combining the above with \eqref{eq:dynamics_rewritten}, we obtain that
$
u=\ll \ba{c|c} 
A_{\Phi}-B_{\Phi}C_{\Phi} & B_{\Phi} \\ \hline
C_Q(\Pi_2-\Pi_1 C_\Phi ) & C_Q \Pi_1
\ea \rr q,
$
which is precisely the same expression as the controller described in \cite{poset_h2_journal}[Theorem 3]. Since this corresponds to the optimal $\mc{H}_2$ controller, we have the required result.
\end{proof}

\end{document}